\documentclass[11pt,twoside]{preprint}
\usepackage{amsmath,amsthm,amssymb,enumerate}
\usepackage{hyperref}
\usepackage{mathrsfs}
\usepackage{breakurl}
\usepackage{appendix}
\usepackage{xcolor}
\usepackage{mhequ}
\usepackage{comment}
\usepackage{times}
\usepackage{microtype}
\usepackage[a4paper,innermargin=1in,outermargin=1.8in,bottom=1.5in,marginparwidth=1.5in,marginparsep=3mm]{geometry}

\usepackage{stmaryrd}

\renewcommand{\theequation}{\thesection.\arabic{equation}}

\newtheorem{theorem}{Theorem}[section]
\newtheorem{lemma}[theorem]{Lemma}
\newtheorem{proposition}[theorem]{Proposition}
\newtheorem{corollary}[theorem]{Corollary}

\theoremstyle{definition}
\newtheorem{assumption}[theorem]{Assumption}
\newtheorem{definition}[theorem]{Definition}

\theoremstyle{remark}
\newtheorem{remark}[theorem]{Remark}

\def\scal#1{\langle #1 \rangle}

\allowdisplaybreaks

\newcommand\bH{\mathbb{H}}
\newcommand\bone{\mathbf{1}}

\newcommand\cR{\mathcal{R}}

\newcommand\cK{\mathcal{K}}
\newcommand\cD{\mathcal{D}}

\newcommand\cC{\mathcal{C}}

\newcommand\cS{\mathcal{S}}

\newcommand\cM{\mathcal{M}}

\newcommand\cF{\mathcal{F}}

\newcommand\scB{\mathcal{B}}

\newcommand\scP{\mathcal{P}}

\newcommand\frH{\mathfrak{H}}

\def\eps{\varepsilon}

\newcommand{\R}{\mathbb{R}}
\newcommand{\E}{\mathbb{E}}

\begin{document}
\title{Boundary regularity of stochastic PDEs}
\author{M\'at\'e Gerencs\'er}
\institute{IST Austria, \email{mate.gerencser@ist.ac.at}}
\maketitle
\begin{abstract}
The boundary behaviour of solutions of stochastic PDEs with Dirichlet boundary conditions can be surprisingly - and in a sense, arbitrarily -  bad: 
as shown by Krylov \cite{K_Brown}, for any $\alpha>0$ one can find a simple $1$-dimensional constant coefficient linear equation whose solution at the boundary is not $\alpha$-H\"older continuous.

We obtain a positive counterpart of this: under some mild regularity assumptions on the coefficients, solutions  of semilinear SPDEs on $\cC^1$ domains are proved to be $\alpha$-H\"older continuous up to the boundary with \emph{some} $\alpha>0$.
\end{abstract}

\tableofcontents
\section{Introduction}

We consider semilinear stochastic partial differential equations (SPDEs) on domains (where the assumptions
and precise understanding of the equation is postponed to Section \ref{sec:Formulation}) of the type
\begin{equs}[eq:main]
        d u&=(a^{ij}D_iD_j u+f(u,\nabla u))\,dt+(\sigma^{ik}D_i u+g^k(u))\,dW^k_t\quad & \text{on } &\R_+\times G,\\
        u&=0, & \text{on } &\R_+\times\partial G,\\
        u_0&=\psi&\text{on }& G,
\end{equs}
with the Einstein summation in place. 
The well-posedness in the variational sense of a large class of such equations is known since the 70's (\cite{Pardoux1}, \cite{KR_SEE81}),
 and interior regularity (at least for the linear ones) results are available from the 90's, starting from \cite{K_W2m}, which initiated a series of works, see among others
  \cite{KL_line}, \cite{KL_hspace}, \cite{Lot_Dirichlet}, \cite{Lot_Degen}, \cite{Kim_Lp}, \cite{Kim_Divergence}, see also \cite{Flan_Compatib} for another approach.
Concerning boundary regularity, while the above works give some partial results, the theory is much less satisfactory.
Even in the linear case, the rather natural question whether the solution is continuous up to the boundary
(and therefore whether the boundary condition is actually satisfied in the classical sense)
 has remained in general unanswered, no matter how smooth the coefficients and the boundary of the domain are,
``\emph{[becoming] a major challenge for the theory}'' according to Krylov \cite{K_Survey}. 
Part of the reason why analysing solutions near the boundary is problematic is the fact that the boundary behaviour is indeed quite bad, as illustrated by the following result. Recall that if in the formulation below the coefficient in the noise were greater than $\sqrt{2}$, then the equation would become ill-posed.
\begin{theorem}[\cite{K_Brown}]
There exists a $\lambda_0>0$ such that if $0<\lambda<\lambda_0$, $\psi\in\cC^\infty_0(\R_+)$ is non identically 0, and $u$ denotes the solution of 
\begin{equs}
        d u&=D^2u\,dt+\sqrt{2-\lambda}D u\,dW_t\quad & \text{on } &\R_+\times \R_+,\\
        u&=0, & \text{on } &\R_+\times\{0\},\\
        u_0&=\psi&\text{on }& \R_+,
\end{equs}
then almost surely there exists a dense subset
$S\subset\R_+$ such that for all $s\in S$ and $\alpha>e^{-\tfrac{1}{2\lambda}}$
\begin{equ}
\lim_{x\downarrow0}u_s(x)x^{-\alpha}=\infty.
\end{equ}
\end{theorem}

The main goal of the present article is to prove that solutions of \eqref{eq:main} \emph{are} H\"older-continuous up to the boundary, with some exponent. In light of the above, this exponent of course has to depend on the equation itself, and as we will see, this dependence is in fact only on through a few parameters of the linear part of the equation. Since the precise statement requires a bit of technical setup, we postpone it to the next section, see Theorem \ref{thm:main}. Our proof is inspired by \cite{K_Onemore}, where the particular case of $d=1$, $f=g=\nabla a=\nabla \sigma=0$, was treated. Importantly, unlike the above mentioned `partial' results, its approach relied neither on a `smallness' nor on a `compatibility' condition on $\sigma$.

To our best knowledge the most general well-posedness results for \eqref{eq:main} use the variational theory, which however strongly restricts the growth of $f$. We prove a more general existence and uniqueness result in Theorem \ref{thm:WP}. That itself requires no growth assumption at all on $f(u,\nabla u)$ in $u$, and this allows us to state also Theorem \ref{thm:main} under mild (arbitrary polynomial) growth conditions.

The article is organised as follows. In the following section, after setting up most of the notations, the main result is stated, which is followed by the aforementioned solvability result in Section \ref{sec:WP}, and the rest of the paper is devoted to the proof of Theorem \ref{thm:main}. The proof has four main components: Reducing the problem to equations with linear structure and more regular data, transforming the simplified equation to a PDE with random coefficients on a random domain, establishing certain geometric properties of this random domain, and finally using these properties to prove the appropriate decay at the boundary. Section \ref{sec:proof} is structured according to these steps.

\section{Formulation}\label{sec:Formulation}
Fix a complete filtered probability space $(\Omega,(\cF_t)_{t\geq0},P)$ carrying an infinite sequence of independent Wiener processes $(W^k_t)_{k\in\mathbb{N},\,t\geq0}$. 
The predictable $\sigma$-algebra on $\Omega\times\R_+$ is denoted by $\scP$.
Whenever expectations are taken with respect to a different probability measure $\hat P$, it will be denoted by $\E^{\hat P}$. 
Let us also fix $T>0$.
Given a $d$-dimensional stochastic differential equation (SDE),
\begin{equation}\label{eq:flows explanation}
dX^i_t=\alpha^i(X_t)\,dt+\beta^{ik}(X_t)\,dW^k_t,\quad i=1,2,...,d
\end{equation}
driven by $W$, the corresponding stochastic flow on $[0,T]$ is a continuous random field $(X_{s,t}(x))_{0\leq s\leq t\leq T,x\in\R^d}$ such that for all $s$ and $x$,
 the process $(X_{s,t}(x))_{ s\leq t\leq T}$ is a solution of the equation \eqref{eq:flows explanation} with initial condition $X_{s,s}(x)=x$,
  and that furthermore almost surely for all $0\leq s\leq t\leq v\leq T$ and $x\in\R^d$, the identity $X_{t,v}(X_{s,t}(x))=X_{s,v}(x)$ holds. When the stochastic differential in \eqref{eq:flows explanation} is replaced by the backward It\^o differential $d\overleftarrow{W}_t$, then one can correspondingly talk about the backward flow $(X_{t,s}(x))_{0\leq s\leq t\leq T,x\in\R^d}$. Often it turns out that for any $0\leq s\leq t\leq T$, $X_{s,t}(\cdot)$ is a diffeomorphism from $\R^d$ to itself, in which case one can talk about the inverse flow $(X_{s,t}^{-1}(x))_{0\leq s\leq t\leq T,x\in\R^d}$.

By $B_r(x)$ we understand the $d$-dimensional ball of radius $r\geq0$ around $x\in \R^d$, and for $x=0$ the $x$ argument is often dropped. We denote by $\langle\cdot,\cdot\rangle$ the scalar product in $\R^d$. The distance between two closed sets $A$ and $B$ is denoted by $d(A,B)$. The Borel $\sigma$-algebra on $\R^n$ is denoted by $\scB(\R^n)$.

We fix a bounded $\cC^1$-domain $G\subset\R^d$ (as defined in e.g., \cite{Kim_Lp}), denote $G^c=\R^d\setminus G$, $Q=[0,T]\times G$,
$G^+=G+B_1:=\{x\in\R^d:\,\exists x_1\in G,\,x_2\in B_1:\, x=x_1+x_2\}$,
$Q^+=[0,T]\times G^+$, and for $T_0\geq 0$, $Q_{T_0}=[T_0,T]\times G$. Fix a $\cC^\infty$ function $\Psi$ defined on $G$ such that for all $x\in G$,
\begin{equ}
d(x,\partial G)\leq N\Psi(x)\leq N'd(x,\partial G),\quad d(x,\partial G)^{|k|}|D_k\nabla\Psi(x)|\leq N(k)
\end{equ}
for some constants $N$, $N'$, $(N(k))$, $k$ running over all possible multiindices. For the existence of such function, see e.g. \cite{Lot_WeightedSpaces}.

Derivatives in the direction of the $i$-th unit direction in $\R^d$ are denoted by $D_i$. By $\nabla$ we denote the gradient, with the convention that for $f:\R^d\rightarrow\R^k$, $(\nabla f)^{ij}=D_jf^i$.

For $\gamma\in \R$ and $p\geq 1$, by $H^\gamma_p=H^\gamma_p(G)$ we mean the usual Sobolev spaces,
see e.g. \cite{Triebel}. By $\dot H^\gamma_p$ we mean the closure of $\cC^\infty_0(G)$ in the $H^\gamma_p$ norm. For $\gamma,\theta\in\R$ and $p\geq 1$, by $H^\gamma_{p,\theta}=H^\gamma_{p,\theta}(G)$ we understand weighted Sobolev spaces. An easily accessible definition of them is to first set for $\gamma=n\in\mathbb{N}$,
\begin{equ}\label{eq:weighted spaces}
\|u\|_{H^\gamma_{p,\theta}}^p:=\sum_{i=0}^n\sum_{|\alpha|=i}\int_G|D_{\alpha_1}\cdots D_{\alpha_i} u|^p(x)d(x,\partial G)^{\theta-d+ip}\,dx,
\end{equ}
and then extend this scale of spaces to noninteger and nonnegative values of $\gamma$ by interpolation and duality, respectively. See \cite{Lot_WeightedSpaces} and \cite{K_Traces} for more details, and also for a more intrinsic equivalent definition of these spaces.

H\"older spaces $\cC^\alpha(A)$ on some set $A\subset\R^n$ for $\alpha\in(0,1]$ are defined with the norm
\begin{equ}
\|u\|_{\cC^\alpha(A)}:=\sup_{x\in A}|u(x)|+\sup_{x\neq y\in A}\frac{|u(x)-u(y)|}{|x-y|^\alpha}.
\end{equ}
For $\alpha>0$, $u\in\cC^\alpha$ if all of its $k$-th derivatives, $|k|<\lceil \alpha\rceil$, belong to $\cC^{\alpha+1-\lceil \alpha\rceil}$.

All of the above spaces can easily be extended to $l_2$-valued (or $(l_2)^n$-valued, for that matter) functions, by taking the appropriate operations coordinate-wise and replace the absolute value by the $l_2$-norm. Hence the dimension of the function spaces will not always be detailed - for example, the reader understands that requiring the coefficient $\beta$ of an equation like \eqref{eq:flows explanation} to be of class $\cC^1$ is to require it to be an element of $\cC^1(\R^d,(l_2)^d).$

The understanding of the solution of \eqref{eq:main} is the following.
\begin{definition}
A solution of \eqref{eq:main} is a continuous adapted $L_2$-valued process $u$ that furthermore belongs to $ L_\infty(Q)\cap L_2([0,T],\dot H^1_2(G))$ almost surely, such that
for all $\varphi\in\cC^\infty_0(G)$ the identity 
\begin{equs}[eq:nondiv]
(u_t,\varphi)=(\psi,\varphi)&+\int_0^t(-D_j u_s,a_s^{ij}D_i\varphi)+(f_s(u_s,\nabla u_s)-(D_ia_s^{ij})D_ju_s,\varphi)\,ds
\\&+\int_0^t(\sigma_s^{ik}D_iu_s+g_s^k(u_s),\varphi)\,dW^k_s
\end{equs}
holds almost surely for all $t\in[0,T]$, where $(\cdot,\cdot)$ denotes the $L_2$-inner product.
\end{definition}

Our assumptions for the main result are as follows (in particular, they are more than sufficient to guarantee that all expressions in \eqref{eq:nondiv} make sense).
\begin{assumption}\label{as:coercivity}
There exists a $\kappa>0$ such that for all $(t,\omega,x)\in[0,T]\times\Omega\times(G+B_{1/2}),$
\begin{equ}
\bar a:=a-\tfrac{1}{2}\sigma\sigma^*\geq \kappa I
\end{equ}
holds in the sense of positive semidefinite matrices.
\end{assumption}

\begin{assumption}\label{as:coeff regularity}
(a) The coefficients $a$ and $\sigma$ are $\scP\otimes \scB(\R^d)$-measurable functions that vanish outside $G^+$.
There exist constants $K>0$ and $\nu\in(0,1)$ such that for all $t$ and $\omega$, 
\begin{equ}
\|a_t(\cdot)(\omega)\|_{\cC^{2+\nu}(\R^d)}+\|\sigma_t(\cdot)(\omega)\|_{\cC^{3+\nu}(\R^d)}\leq K.
\end{equ}

(b) There exists a random variable $H$ with finite moments of all order such that for all $\omega$,
\begin{equ}
\|\sigma_\cdot(\cdot)(\omega)\|_{\cC^\nu([0,T],L_\infty(\R^d))}\leq H(\omega).
\end{equ}
\end{assumption}

\begin{assumption}\label{as:nonlinearity}
(a) The function $f(u,\nabla u)$ takes the form $f(u,\nabla u)=\bar f(u)+\nabla\cdot(\hat f(u))$, with $\hat f(0)=0$.
The functions $\bar f$, $\hat f$, and $g$ are $\scP\otimes\scB(\R^d)\otimes\scB(\R)$-measurable, with values in $\R$, $\R^d$, and $l_2$, respectively, that vanish outside $G^+$.
The real-valued function $\psi$ is $\cF_0\otimes\scB(\R^d)$-measurable and vanishes outside $G^+$.
 The function $\bar f_t(x,y)(\omega)$ is continuous in $y\in\R$ uniformly in $t,x,\omega$, and there exists a constant $K>0$ such that
\begin{equs}
(y-y')(\bar f_t(x,y)(\omega)-\bar f_t(x,y')(\omega))&\leq K|y-y'|^2
\\
|\hat f_t(x,y)(\omega)-\hat f_t(x,y')(\omega)|&\leq K|y-y'|
\\
|g_t(x,y)(\omega)-g_t(x,y')(\omega)|&\leq K|y-y'|
\end{equs}
for all $t,x,y,y',\omega$.

(b) There exists a constant $m>0$ such that for all $t,x,y,\omega$, 
\begin{equ}
|\bar f_t(x,y)(\omega)-\bar f_t(x,0)(\omega)|\leq K|y|^m.
\end{equ}
\end{assumption}

\begin{assumption}\label{as:data regularity}
The functions $\psi$, $f^0=f^0_t(x):=\bar f_t(x,0)$, and $g^0:=g_t(x)=g_t(x,0)$ satisfy, for some $\bar\nu>0$ and for all $p\in[2,\infty)$
\begin{equs}
\E\Big(\|\psi\|_{H^{\bar \nu}_p}&+\|f^0\|_{L_{d+4}([0,T],H^{-1+\bar\nu}_{d+4})}
+\|f^0\|_{L_p([0,T],H^{-2+\bar\nu}_{p,d-2+2p})}
\\
&+\|g^0\|_{L_{d+4}([0,T],H^{\bar\nu}_{d+4,d-1/2})}+
\|g^0\|_{L_p([0,T],H^{-1+\bar\nu}_{p,d-2+p})}\Big)^2<\infty.
\end{equs}
\end{assumption}

Let us finally denote $d_1:=\inf\{k\in\mathbb{N}:\,\sigma^{il}_t(x)(\omega)\equiv 0\,\forall l>k\}$.

These assumptions, unless one assumes further control of the growth of $\bar f$ in $u$, are not quite enough to fit in the $L_2$-theory (\cite{Pardoux1}, \cite{KR_SEE81}), 
and in fact as far as the author is aware, no result on well-posedness in this scope is known. In the next section we prove some existence and uniqueness results that well cover the above setting. The main result of the paper then reads as follows.

\begin{theorem}\label{thm:main}
Let Assumptions \ref{as:coercivity} and \ref{as:coeff regularity} hold and suppose $d_1<\infty$. Then there exists an $\alpha=\alpha(\kappa,K,\bar\nu,d,d_1)>0$ such that for any $T_0>0$ and $\psi,f,g$ satisfying Assumptions \ref{as:nonlinearity} and \ref{as:data regularity}, a unique solution $u$ of \eqref{eq:main} exists and
almost surely
\begin{equ}
\sup_{(t,x)\in Q_{T_0}}|u(t,x)|d(x,\partial G)^{-\alpha}<\infty.
\end{equ}
Moreover, for fixed $K,\bar\nu,d,d_1$, there exists a $c_0$ such that for sufficiently small $\kappa$ one has $\alpha>e^{-c_0/\kappa}$.
\end{theorem}
\begin{remark}\label{rem:elso}
Since $\psi$ is not assumed to vanish at the boundary, one can in general not take $T_0=0$.
Concerning the assumption $d_1<\infty$, one could actually do slightly better with essentially the same argument, see Remark \ref{rem:d1} below.
\end{remark}

\begin{remark}
Assumption \ref{as:data regularity} is somewhat cumbersome. A stronger, but perhaps more tractable condition would be
\begin{equ}\label{eq:data reg2}
\E\Big(\|\psi\|_{H^{\tilde \nu}_{\tilde p}}+\|f^0\|_{L_\infty([0,T],H^{-1+\tilde\nu}_{\tilde p})}+\|\Psi^{-1/(2(d+4))}g^0\|_{L_\infty([0,T],H^{\tilde \nu}_{\tilde p})} \Big)^2<\infty
\end{equ}
with some fixed $\tilde\nu>0$, $\tilde p>d/\tilde\nu$. As one can see from the basic properties of weighted Sobolev spaces (which we recall in Subsection \ref{subsec:simpl}), \eqref{eq:data reg2} implies Assumption \ref{as:data regularity}, with $\bar \nu=\tilde\nu-d/\tilde p$.
One reason why one would not want to impose \eqref{eq:data reg2}, however, is that it assumes some pointwise decay at the boundary from $g^0$, while Assumption \ref{as:data regularity} does not.
\end{remark}

Combining Theorem \ref{thm:main} and some interior regularity, one easily gets the following corollary, which is proved in Subsection \ref{subsec:simpl}.
\begin{corollary}\label{cor}
Assume the setting of Theorem \ref{thm:main} and let $\hat\alpha$ satisfy
\begin{equ}
0<\hat\alpha<\frac{\alpha\bar\nu}{3(\alpha+\bar\nu)}
\end{equ}
Then for any $T_0>0$, the solution $u$ of \eqref{eq:main} belongs to $\cC^{\hat\alpha}(\overline{Q_{T_0}})$ almost surely.
\end{corollary}
\section{Existence and uniqueness of the solution}\label{sec:WP}
First we state the existence result, under some reduced regularity and growth assumptions. Note that we momentarily switch to equations in divergence form, but since in the rest of the article the regularity condition Assumption \ref{as:coeff regularity} on the coefficients will be in place, switching between divergence and non-divergence form equations is harmless. We also remark that for Theorem \ref{thm:WP} one in fact only needs $G$ to be a Lipschitz domain.
\begin{assumption}\label{as:data regularity WP}
The functions $\psi$, $f^0$, and $g^0$ satisfy, for some $\mu>0$, 
\begin{equ}
\cK_{0}:=\|\psi\|_{L_\infty(G)}+\|f^0\|_{L_{d+2+\mu}([0,T],H^{-1}_{d+2+\mu})}+\|g^0\|_{L_{d+2+\mu}(Q)}<\infty
\end{equ}
almost surely.
\end{assumption}
Define also
\begin{equ}
\cK_{1}:=\|\psi\|_{L_2(G)}+\|f^0\|_{L_{2}([0,T],H^{-1}_{2})}+\|g^0\|_{L_{2}(Q)}.
\end{equ}
\begin{theorem}\label{thm:WP}
Let Assumptions \ref{as:coercivity}, \ref{as:nonlinearity} (a), and \ref{as:data regularity WP} hold and assume that $a$ and $\sigma$ are $\scP\otimes \scB(\R^d)$-measurable functions bounded by $K$. 
Then there exists a unique continuous $L_2$-valued adapted process $u$ that furthermore belongs to $ L_\infty(Q)\cap L_2([0,T],\dot H^1_2(G))$ such that
for all $\varphi\in\cC^\infty_0(G)$ the identity 
\begin{equs}[eq:div]
(u_t,\varphi)=(\psi,\varphi)&+\int_0^t(-D_j u_s,a_s^{ij}D_i\varphi)+(f_s(u_s,\nabla u_s),\varphi)\,ds
\\&+\int_0^t(\sigma^{ik}D_iu_s+g_s^k(u_s),\varphi)\,dW^k_s
\end{equs}
holds almost surely for all $t\in[0,T]$. Finally, the estimates
\begin{equs}
\E\|u\|_{L_\infty(Q)}^p&\leq N(\kappa,K,\mu,T,d,G,p)\E\cK_{0}^p,\\
\E\|  u\|_{L_2([0,T],H^{1}_2(G))}^2&\leq N(\kappa,K,T,d,G) \E\cK_{1}^2
\end{equs}
hold with any $p\in(0,\infty)$.
\end{theorem}
\begin{proof}
The proof closely follows those of Theorems 2.1-5.2 in \cite{DG15} and (as, in fact, indicated therein)
one needs only make sure that the nonlinear terms do not change anything essential.
We therefore do not aim to repeat the whole argument, but rather will only detail the verification of this.

Define for $n,m\in\mathbb{N}$,
\begin{equ}
\bar f^{(n,m)}_t(x,y):=\bar f_t(x,-n\vee y\wedge m),\quad f^{(n,m)}(u,\nabla u)=\bar f^{(n,m)}(u)+\nabla\cdot(\hat f(u)).
\end{equ}
Since $\bar f^{(n,m)}$ has linear growth, the results of \cite{KR_SEE81} apply and hence one has the existence of a unique continuous $L_2$ valued adapted process
$u^{(n,m)}$ which furthermore belongs to $L_2([0,T],\dot H^1_2(G))$ and such that \eqref{eq:div} holds with $u^{(n,m)}$ and $f^{(n,m)}$ in place of $u$ and $f$, respectively.

Applying It\^o's formula \cite[Lem~3.2]{DG15} to $\|u_t^{(n,m)}\|_{L_p(G)}^p$, $p\geq2$, one gets
\begin{equs}[eq:ito]
\int_G&|u^{(n,m)}_t|^p\,dx=\int_G|\psi|^p\,dx
\\
&+\int_0^t\int_Gp|u_s^{(n,m)}|^{p-2}u_s^{(n,m)}(\sigma^{ik}_sD_iu_s^{(n,m)}+g_s^k(u_s^{(n,m)}))\,dx\,dW_s^k
\\
&+\int_0^t\int_G -p(p-1)|u_s^{(n,m)}|^{p-2}D_iu_s^{(n,m)}a^{ij}_sD_ju_s^{(n,m)}
\\
&\quad\quad+p|u_s^{(n,m)}|^{p-2}u_s^{(n,m)}\bar f^{(n,m)}_s(u_s^{(n,m)})
\\
&\quad\quad-p(p-1)|u_s^{(n,m)}|^{p-2}\nabla u_s^{(n,m)}\cdot \hat f_s^{(n,m)}(u_s^{(n,m)})
\\
&\quad\quad+(1/2)p(p-1)|u_s^{(n,m)}|^{p-2}|\sigma^{ik}_sD_iu_s^{(n,m)}+g_s^k(u_s^{(n,m)})|_{l_2}^2\,dx\,ds.
\end{equs}
Looking at the contribution of the nonlinear terms, we can write, by Assumption \ref{as:nonlinearity} (a)
\begin{equs}
u_s^{(n,m)}&\bar f^{(n,m)}_s(u_s^{(n,m)})
\leq K|u_s^{(n,m)}|^2+u_s^{(n,m)}f^0_s.
\end{equs}
Recall that by Assumption \ref{as:data regularity WP}, $f^0=\bar h^0+\nabla\cdot\hat h^0$, where $\bar h^0,\hat h^0\in L_{d+2+\mu}(Q)$ and
$\|\bar h^0\|_{L_{d+2+\mu}(Q)}+\|\hat h^0\|_{L_{d+2+\mu}(Q)}\leq2\cK_0<\infty$. Therefore by the above bounds, integration by parts, and Young's inequality we have, for any $\eps>0$,
\begin{equs}
\int_G&p|u_s^{(n,m)}|^{p-2}u_s^{(n,m)}\bar f^{(n,m)}_s(u_s^{(n,m)})\,dx
\\
&\leq \int_G p^2|u_s^{(n,m)}|^{p-2}\Big(K|u_s^{(n,m)}|^2+\eps|\nabla u_s^{(n,m)}|^2+|u_s^{(n,m)}||\bar h^0_s|+C(\eps)|\hat h^0_s|^2\Big)\,dx,
\end{equs}
for some constant $C(\eps)$ depending only on $\eps$ and $K$.
Next, we have
\begin{equ}
|\nabla u_s^{(n,m)}\cdot \hat f_s^{(n,m)}(u_s^{(n,m)})|\leq
\eps|\nabla u_s^{(n,m)}|^2+C(\eps)|u_s^{(n,m)}|^2,
\end{equ}
allowing one to bound the second to last term in \eqref{eq:ito}
As for the contribution of $g$, one simply has
\begin{equs}
(&2\sigma^{ik}_sD_iu_s^{(n,m)}+g_s^k(u_s^{(n,m)}))g^k_s(u_s^{(n,m)})
\\&=
\Big(2\sigma^{ik}_sD_iu_s^{(n,m)}+(g_s^k(u_s^{(n,m)})-(g_s^0)^k)+(g_s^0)^k\Big)
((g^k_s(u_s^{(n,m)})-(g_s^0)^k)+(g_s^0)^k).
\end{equs}
Therefore by Assumption \ref{as:nonlinearity} (a) we have, for any $\eps>0$,
\begin{equs}
\sum_{k\geq 0}&\int_G(1/2)p(p-1)|u_s^{(n,m)}|^{p-2}(2\sigma^{ik}_sD_iu_s^{(n,m)}+g_s^k(u_s^{(n,m)}))g^k_s(u_s^{(n,m)})\,dx
\\&\leq p^2\int_G C(\eps)|u_s^{(n,m)}|^p+C(\eps)|u_s^{(n,m)}|^{p-2}|g^0|^2+\eps|u_s^{(n,m)}|^2|\nabla u_s^{(n,m)}|^2\,dx.
\end{equs}
All of these are of precisely the same order as the contributions coming from the lower order linear terms in \cite{DG15}. Note also that the constants on the right-hand sides do not depend on $n$ and $m$.
 The resulting energy estimates are therefore virtually identical to the ones in the linear case, and thus so is Moser's iteration. One therefore obtains the bounds
\begin{equ}\label{eq:est1}
\E\|u^{(n,m)}\|_{L_\infty(Q)}^p\leq N\E\cK^p_0
\end{equ}
with $N$ depending on $\kappa$, $K$, $\mu$, $T$, $d$, $G$, $p$ but not on $n$ and $m$.
Also, applying It\^o's formula for $\|u_s\|_{L_2(G)}^2$, by the above and Assumption \ref{as:coercivity} one gets
\begin{equs}
\int_G&|u_T^{(n,m)}|^2\,dx\leq\int_G|\psi|^2\,dx+M_T-2\kappa\int_0^T\int_G|\nabla u_s^{(n,m)}|^2\,dx\,ds
\\&\quad+\int_0^T\int_GC(\eps)(|u_s^{(n,m)}|^2+|\bar h^0_s|^2+|\hat h^0_s|^2+|g^0_s|^2)+\eps|\nabla u_s^{(n,m)}|^2\,dx\,ds
\end{equs}
with some martingale $M$. Hence one obtains
\begin{equ}\label{eq:est2}
\E\|  u_s^{(n,m)}\|_{L_2([0,T],H^{1}_2(G))}^2\leq N\E\cK_1^2,
\end{equ}
with $N$ having the same dependencies as before, except for $\mu$ and $p$.

Now we let $n\rightarrow\infty$.
By the comparison principle \cite[Thm~3.3]{DGy14}
one has that $ u^{(n,m)}\leq u^{(n',m)}$ for $n'\geq n$, which, thanks to \eqref{eq:est1}, implies that $ u^{(n,m)}$ not only converges as $n\rightarrow\infty$,
but is in fact constant in $n$ after an index $N=N(\omega)$.
This implies that the limit $ u^{(\infty,m)}$ is a solution of \eqref{eq:div} with $f$ replaced by 
\begin{equ}
f^{(\infty,m)}(u,\nabla u)=\bar f^{(\infty,m)}(u)+\nabla\cdot(\hat f(u)),\quad
\bar f^{(\infty,m)}_t(x,y):=\bar f_t(x, y\wedge m),
\end{equ} 
and moreover, $ u_s^{(\infty,m)}$ also satisfies the bounds \eqref{eq:est1}-\eqref{eq:est2}.
One then passes to the $m\rightarrow\infty$ limit similarly, and the limit $u:= u^{(\infty,\infty)}$ is indeed the solution claimed in the theorem.

As for the uniqueness, take two solutions $u$ and $v$ and write It\^o's formula for $\|e\|_{L_2(G)}^2:=\|u-v\|_{L_2(G)}^2$:
\begin{equs}
\int_G |e_t|^2\,dx&=\int_0^t\int_G-\bar 2a_s^{ij}D_ie_sD_je_s+(u_s-v_s)(f_s(u_s,\nabla u_s)-f_s(v_s,\nabla v_s))
\\&\quad+2\sigma_s^{ik}D_i(u_s-v_s)(g_s^k(u_s)-g_s^k(v_s))+|g_s(u_s)-g_s(v_s)|^2\,dx\,ds+m_t
\end{equs}
with some martingale $m$. By Assumption \ref{as:nonlinearity} (a), one has
\begin{equs}
(u-v)&(f(u,\nabla u)-f(v,\nabla v))
&\leq K|u-v|^2+(u-v)\nabla (\hat f(u)-\hat f(v)).
\end{equs}
After integration by parts, using simply the bound $|g(u)-g(v)|\leq K|u-v|$ in the terms involving $g$, and by Assumption \ref{as:coercivity}, we get
\begin{equ}
\int_G|e_t|^2\leq \int_0^t\int_G-2\kappa|\nabla e_s|^2+C|e_s|^2+C'|e_s||\nabla e_s|\,dx\,ds+m_t
\end{equ}
with some constants $C,C'$ depending on $K$. Hence, Young's inequality, taking expectations, and Gronwall's lemma yields
$(\E\|e_t\|^2_{L_2(G)})_{t\in[0,T]}\equiv 0$. Since $e$ is continuous in $L_2(G)$, $(e_t)_{t\in[0,T]}\equiv 0$ almost surely, as required.
\end{proof}

\section{Proof of Theorem \ref{thm:main}}\label{sec:proof}

\subsection{Simplifying}\label{subsec:simpl}
As a first step, we reduce the statement to a version where the equation is linear, $f$ is regular, and $g$ is simply not present.
To do that, however, we need to derive some further properties of the solution of \eqref{eq:main}, based on $L_p$-theory, and so we recall a couple of notations from it.

We somewhat deviate from the standard convention of the literature in terms of the spaces used, in that the integration exponent in time and in $\omega$ may differ (in fact the latter will mostly be 2), hence the slightly different notation.
Set $U^\gamma_{p,\theta,(q)}=L_q(\Omega,\cF_0,\Psi^{1-2/p}H^{\gamma-2/p}_{p,\theta})$ and let $\bH^\gamma_{p,\theta,(q)}$ be the space of $\scP\otimes\scB(G)$-measurable functions belonging to  $L_q(\Omega, L_p([0,T],H^{\gamma}_{p,\theta}))$.
Let furthermore $\frH^\gamma_{p,\theta,(q)}\subset \Psi\bH^\gamma_{p,\theta,(q)}$ consist of functions $u$ for which
there exists a $\psi\in U^\gamma_{p,\theta,(q)}$, $f\in\Psi^{-1}\bH^{\gamma-2}_{p,\theta,(q)}$, and $g\in\bH^{\gamma-1}_{p,\theta,(q)}$, such that for all $\varphi\in\cC^\infty_0(G)$ the identity
\begin{equ}
(u_t(\cdot),\varphi)=(\psi,\varphi)
+\int_0^t (f_s(\cdot),\varphi)\,ds+\int_0^t(g^k_s(\cdot),\varphi)\,dW^k_s
\end{equ}
holds almost surely for all $t\in[0,T]$. We use the norm
\begin{equ}
\|u\|_{\frH^\gamma_{p,\theta,(q)}}=\|\Psi^{-1}u\|_{\bH^{\gamma}_{p,\theta,(q)}}+\|\psi\|_{U^\gamma_{p,\gamma,(q)}}+\|\Psi f\|_{\bH^{\gamma-2}_{p,\theta,(q)}}+\|g\|_{\bH^{\gamma-1}_{p,\theta,(q)}}.
\end{equ}
Let us recall some useful properties of these spaces. First of all, $\Psi^{-\alpha}$ is an isomorphism from $H^\gamma_{p,\theta}$ to $H^\gamma_{p,\theta+\alpha p}$.
The following property, while we did not find explicitly stated elsewhere, follows easily from the definition \eqref{eq:weighted spaces}, interpolation, and duality.
\begin{equs}[eq:easy embed]
H^\gamma_{p,\theta}&\subset H^\gamma_p\quad&&\text{if }\theta\leq d-(\gamma\vee0) p
\\
H^\gamma_{p}&\subset H^\gamma_{p,\theta}\quad&&\text{if }\theta\geq d-(\gamma\wedge0) p
\end{equs}

Finally, invoke from \cite[Thm~4.7]{K_Traces}
that for any $r'\geq r\geq 2$, $\kappa\in[0,1]$, $2/r<\beta\leq 1$, $q\in[0,r]$, $\theta\in\R$, and $\gamma\in\R$, one has the continuous embedding
\begin{equs}\label{eq:embed cont}
\frH^\gamma_{p,\theta,(q)}&\subset \Psi^{1-\gamma+(d-\theta)/p} L_{q}(\Omega,\cC^{\alpha/2-1/p}([0,T],\cC^{\gamma-\beta -d/p}(G))),
\end{equs}
provided 
\begin{equ}
2/p<\alpha<\beta\leq 1,\quad \gamma-\beta-d/p\in(0,1).
\end{equ}

The following is a particular case of the of the quite general $L_p$-theory for SPDEs on domains from \cite[Thm~2.9]{Kim_Lp}.
\begin{theorem}\label{thm:Kim}
Let Assumption \ref{as:coeff regularity} (a) hold, and assume that $f$ and $g$ do not depend on $u$ or $\nabla u$. 
Suppose furthermore that for some $c\in(0,1]$, Assumption \ref{as:coercivity} hold with $\kappa=cK$ and fix $p\geq 2$ and $\theta\in\R$ that satisfy
\begin{equ}\label{eq:theta}
d-1+p[1-\frac{1}{p(1-c)+c}]<\theta<d-1+p.
\end{equ}
Let $q\in[0,p]$, $\gamma\in[0,4]$, and assume $\psi\in U^\gamma_{p,\theta,(q)}$, $f^0\in\Psi^{-1}\bH^{\gamma-2}_{p,\theta,(q)}$, and $g\in\bH^{\gamma-1}_{p,\theta,(q)}$.
Then the solution $u$ of \eqref{eq:main} belongs to $\frH^{\gamma}_{p,\theta,(q)}$
\begin{equ}
\|u\|_{\frH^\gamma_{p,\theta,(q)}}\leq N(\|\psi\|_{U^\gamma_{p,\theta,(q)}}+\|\Psi f\|_{\bH^{\gamma-2}_{p,\theta,(q)}}+\|g\|_{\bH^{\gamma-1}_{p,\theta,(q)}}),
\end{equ}
where $N$ depends on $\kappa$, $K$, $d$, $T$, $G$, $\theta$, $p$, and $q$.
\end{theorem}

\begin{remark}
Notice that \eqref{eq:theta} is always satisfied if $d-2+p\leq\theta<d-1+p$.
\end{remark}
\begin{remark}
Both \eqref{eq:embed cont} and Theorem \ref{thm:Kim} are actually only stated in the references for the $q=p$ case. However, one can easily deduce the $q<p$ case using Lenglart's inequality, see \cite[Prop~IV.4.7]{RYor}.
\end{remark}

\begin{theorem}\label{thm:intcont}
Let Assumptions \ref{as:coercivity}, \ref{as:coeff regularity} (a), \ref{as:nonlinearity}, and \ref{as:data regularity} hold and let $u$ be the solution of \eqref{eq:main} obtained from Theorem \ref{thm:WP}. Then for any $p\in[2,\infty)$, $u\in\frH^{\bar \nu}_{p,d-2+p,(2)}$, and in particular $u$ is a continuous random field in $Q$.
\end{theorem}
\begin{proof}
By Theorem \ref{thm:WP}, and \eqref{eq:easy embed}, one has, for any $p\in[2,\infty)$ 
\begin{equs}
\nabla\cdot \hat f(u) \in L_2(\Omega, L_p([0,T],H^{-1}_{p}))&\subset L_2(\Omega,L_p([0,T],H^{-1}_{p,d+p}))
\\
&\subset L_2(\Omega,L_p([0,T],\Psi^{-1}H^{-1}_{p,d-2+p}))
\\&
\subset\Psi^{-1}\bH^{-2+\bar\nu}_{p,d-2+p,(2)}.
\end{equs}
By similar reasoning and Assumption \ref{as:nonlinearity} (b),
\begin{equ}
(\bar f(u)-f^0)\in L_2(\Omega, L_p([0,T],L_p))\subset
L_2(\Omega, L_p([0,T],H^0_{p,d}))\subset\bH^{-1+\bar \nu}_{p,d-2+p,(2)}.
\end{equ}
Invoking also Assumption \ref{as:data regularity}, one can therefore conclude
that $f(u,\nabla u)\in\Psi^{-1}\bH^{-2+\bar\nu}_{p,d-2+p,(2)}$. 
A similar (in fact, easier) argument  shows that $g(u)\in\bH^{-1+\bar\nu}_{ p,d-2+p,(2)}$.
Also, 
\begin{equ}
\psi\in L_2(\Omega,H^{\bar\nu}_{p})\subset U^{\bar\nu}_{\bar p,d-2+\bar p,(2)}
\end{equ} 
by Assumption \ref{as:data regularity} and \eqref{eq:easy embed}.
Viewing $f(u,\nabla u)$ and $g(u)$ as fixed free terms, we can apply Theorem \ref{thm:Kim},
with $\bar \nu$ in place of $\gamma$, and $2$ in place of $q$, to obtain $u\in\frH^{\bar\nu}_{ p,d-2+ p,(2)}$ as claimed.
The second claim in the theorem follows by simply using \eqref{eq:embed cont}.
\end{proof}
Now that the basic interior regularity is quantified, we can prove Corollary \ref{cor}.

\emph{Proof of Corollary \ref{cor}}.
Let $(t,x),(s,y)\in Q_{T_0}$ and denote $|(t,x)-(s,y)|=\eps$, $d(x,\partial G)\vee d(y,\partial G)=\delta$. From Theorem \ref{thm:intcont} and \eqref{eq:embed cont} we can deduce that
for any $\bar\alpha<\bar\nu/3$ one has
\begin{equ}
|u_t(x)-u_s(y)|\leq \eta_0 \delta^{-\bar\nu}\eps^{\bar\alpha}.
\end{equ}
with some random variable $\eta_0$. Theorem \ref{thm:main}, on the other hand, yields that
\begin{equ}
|u_t(x)-u_s(y)|\leq \eta_1\delta^\alpha
\end{equ}
with some random variable $\eta_1$. If $\delta\geq\eps$, then this already gives the desired H\"older estimate.
Otherwise set $\lambda=\alpha/(\alpha+\bar\nu)\in(0,1)$, and note that combining the two above bounds give
\begin{equ}
|u_t(x)-u_s(y)|\leq\eta_0^\lambda\eta_1^{1-\lambda}\eps^{\lambda\bar\alpha},
\end{equ}
as required.
\qed

Introduce for $ {T_0}\geq0$, $\alpha\geq 0$, the spaces $L_{\infty,\alpha}(Q_{T_0})$, of functions $u\in L_\infty(Q_{T_0})$ such that
\begin{equ}
\|u\|_{L_{\infty,\alpha}(Q_{T_0})}:=\sup_{(t,x)\in Q_{{T_0}}}|u(t,x)|d(x,\partial G)^{-\alpha}<\infty.
\end{equ}
It is easy to check that under the complex interpolation $[\cdot,\cdot]_{\theta}$ (for its definition see e.g. \cite{Triebel}) these spaces behave as expected.
\begin{proposition}
Let $\alpha\neq\alpha'$, $\theta\in(0,1)$, and ${T_0}\geq 0$. Then 
\begin{equ}
L_{\infty,(1-\theta)\alpha+\theta\alpha'}(Q_{T_0})=[L_{\infty,\alpha}(Q_{T_0}),L_{\infty,\alpha'}(Q_{T_0})]_{\theta}.
\end{equ}
\end{proposition}
\begin{proof}
Denote by $l_\infty^\alpha(L_\infty)$ the set of sequences with elements from $L_\infty(Q_{T_0})$ such that
\begin{equ}
\|(f_n)_{n\geq 0}\|_{l_\infty^\alpha(L_\infty)}=\sup_{n\geq 0}2^{\alpha n}\|f_n\|_{L_\infty(Q_{T_0})}<\infty.
\end{equ}
Then the linear operators $S:\,L_{\infty,\alpha}(Q_{T_0})\rightarrow l_\infty^\alpha(L_\infty)$, $R:\,l_\infty^\alpha(L_\infty)\rightarrow L_{\infty,\alpha}(Q_{T_0})$
\begin{equ}
(S u)_n(t,x):=\bone_{d(x,\partial G)\in [2^{-n-1},2^{-n}]\cdot\text{diam}(G)}u_t(x),\quad
(R (f))_t(x):=\sum_{n\geq 0}f_n(t,x)
\end{equ}
are bounded and satisfy $RS=\text{id}$. The interpolation properties of the spaces $l_\infty^\alpha(L_\infty)$ (see \cite[Thm~1.18.2]{Triebel}) then imply the claim, by \cite[Thm~1.2.4]{Triebel}.
\end{proof}

The setting of the aforementioned simpler version of Theorem \ref{thm:main} is then as follows.
\begin{assumption}\label{as:special}
The function $f$ does not depend on $u$ and $\nabla u$, $g=0$, and almost surely
\begin{equ}
\cK_2:=\|\psi\|_{H^1_{d+3}}+\|f^0\|_{L_\infty([0,T],H^1_{d+3})}<\infty.
\end{equ}
\end{assumption}
\begin{theorem}\label{thm:special}
Let Assumptions \ref{as:coercivity} and \ref{as:coeff regularity} hold and suppose $d_1<\infty$. Then there exists an $\alpha=\alpha(\kappa,K,d,d_1)>0$ 
such that for any $T_0>0$ there exists an almost surely finite random variable $\eta_{T_0}$ such that for all $\psi,f,g$ satisfying Assumption \ref{as:special}, the unique solution $u$ of \eqref{eq:main} belongs to $L_{\infty,\alpha}(Q_{T_0})$, and one has the bound
\begin{equ}
\|u\|_{L_{\infty,\alpha}(Q_{T_0})}\leq\eta_{T_0}\cK_2.
\end{equ}
Moreover, for fixed $K,d,d_1$, there exists a $c_0$ such that for sufficiently small $\kappa$ one has $\alpha>e^{-c_0/\kappa}$.
\end{theorem}
\begin{lemma}\label{lem:simplify}
Theorem \ref{thm:special} implies Theorem \ref{thm:main}.
\end{lemma}
\begin{proof}
We only detail that the existence of the positive decay exponent $\alpha$ in Theorem \ref{thm:special} implies the corresponding statement in Theorem \ref{thm:main}, the analogous implication concerning the exponential lower bound follows very similarly.

Fix $T_0>0$ and set, for $c\in[0,\infty]$, $\Omega_c:=\{\eta_{T_0}\leq c\}$.
Let, for $C\geq 1$ and $c\in[0,\infty]$, denote by $\cS^C_c(\bar\psi,\bar f,\bar g)$ the random field $v\bone_{\Omega_c}$, where $v$ solves
\begin{equs}
        d v&=(Ca^{ij}D_iD_j v+\bar f)\,dt+(\sigma^{ik}D_i v+\bar g^k)\,dW^k_t\quad & \text{on } &\R_+\times G,\\
        v&=0, & \text{on } &\R_+\times\partial G,\\
        v_0&=\bar \psi&\text{on }& G.
\end{equs}
When $C=1$ and/or $c=\infty$, the corresponding index will be dropped. 

Theorem \ref{thm:special} implies that for any $c<\infty$, $\cS_c(\bar \psi,\bar f,0)$ is bounded as an operator 
\begin{equ}
\text{from}\quad L_\infty(\Omega,H^1_{d+3})
\times L_\infty(\Omega,L_\infty([0,T],H^1_{d+3}))
\quad\text{to}\quad
L_\infty(\Omega,L_{\infty,\alpha}(Q_{T_0})).
\end{equ}
Theorem \ref{thm:WP} implies that $\cS(\bar \psi,\bar f,0)$ (and hence obviously also $\cS_c(\bar \psi,\bar f,0)$ for any $c<\infty$) is bounded 
\begin{equ}
\text{from}\quad L_p(\Omega,L_\infty(G))
\times L_p(\Omega,L_{d+3}([0,T],H^{-1}_{d+3}))
\quad\text{to}\quad
L_p(\Omega,L_{\infty,0}(Q_{T_0})),
\end{equ}
for any $p\in(0,\infty)$. Hence, by interpolation, $\cS_c(\bar \psi,\bar f,0)$ is also bounded 
\begin{equ}
\text{from}\quad L_2(\Omega,H^\gamma_{d+4})
\times L_2(\Omega,L_{d+4}([0,T],H^{-1+\gamma}_{d+4}))
\quad\text{to}\quad
L_2(\Omega,L_{\infty,\alpha'}(Q_{T_0})),
\end{equ}
where $\gamma\leq \bar \nu\wedge 1/(4(d+4))$, $\alpha'>0$ depends only on $\alpha$, $\bar\nu$, and $d$.
Note now that one has the identity
\begin{equ}\label{eq:id solution map}
\cS_c(\bar \psi,\bar f,\bar g)=\cS_c(\bar \psi,\bar f+(C-1)a^{ij}D_iD_j(\cS^{C}(0,0,\bar g)),0)+\cS_c^{C}(0,0,\bar g).
\end{equ}
By Theorem \ref{thm:Kim}, for sufficiently large $C=C(d)$, $S^{C}(0,0,\bar g)$ is bounded
\begin{equ}
\text{from}\quad \bH^{\gamma}_{d+4,d-1/2,(2)}
\quad\text{to}\quad
\frH^{1+\gamma}_{d+4,d-1/2,(2)}.
\end{equ}
Notice that 
\begin{equs}
\frH^{1+\gamma}_{d+4,d-1/2,(2)}\subset\Psi\bH^{1+\gamma}_{d+4,d-1/2,(2)}
&=\bH^{1+\gamma}_{d+4,-4-1/2,(2)}
\\&\subset L_2(\Omega,L_{d+4}([0,T],H^{1+\gamma}_{d+4})),
\end{equs}
where for the last inclusion we used \eqref{eq:easy embed} and the condition on $\gamma$. It is known (see \cite[Thm~3.1]{Lot_WeightedSpaces}) that $a^{ij}D_iD_j$ maps $H^{1+\gamma}_{d+4}$ to $H^{-1+\gamma}_{d+4}$. Therefore the first term in \eqref{eq:id solution map} is bounded
\begin{equs}[eq:bd]
\text{from}\quad L_2(\Omega,H^\gamma_{d+4})
\times L_2(\Omega,L_{d+4}&([0,T],H^{-1+\gamma}_{d+4}))\times\bH^{\gamma}_{d+4,d-1/2,(2)}
\\&\quad\text{to}\quad
L_2(\Omega,L_{\infty,\alpha'}(Q_{T_0})).
\end{equs}
Finally, \eqref{eq:embed cont} implies that for a sufficiently small $\eps=\eps(d,\gamma)>0$, $\frH^{1+\gamma}_{d+4,d-1/2,(2)}$ is embedded into $\Psi^\eps L_2(\Omega, L_\infty(Q))$, and thus (possibly after lowering the value of $\alpha'$) the whole solution map $\cS_c(\bar \psi,\bar f,\bar g)$ has boundedness in property in \eqref{eq:bd}.

Since on $\Omega_c$, $u=\cS(\psi,f(u,\nabla u),g(u))$, and by assumption $\psi\in L_2(\Omega,H^\gamma_{d+4})$, it suffices to check that
\begin{equ}\label{eq:12}
f(u,\nabla u)\in L_2(\Omega,L_{d+4}([0,T],H^{-1+\gamma}_{d+4}(G))),\quad g(u)\in\bH^{\gamma}_{d+4,d-1/2,(2)}.
\end{equ}
The first of these inclusions already follows from Theorem \ref{thm:WP}:
by assumption $f^0\in L_2(\Omega,L_{d+4}([0,T],H^{-1+\gamma}_{d+4}(G)))$, we have already seen that $|\bar f(u)-f^0|\leq K|u|^m\in L_2(\Omega,L_\infty(Q))$, and note that
\begin{equ}
\nabla\cdot \hat f(u)\in L_2(\Omega,L_2(Q))\cap L_2(\Omega,L_\infty([0,T],H^{-1}_\infty(G))
\end{equ}
implies, by interpolation,  $\nabla\cdot\hat f( u) \in L_2(\Omega,L_{d+4}([0,T],H^{-1+\gamma}_{d+4}(G)))$.
The second inclusion on \eqref{eq:12} is a consequence of the Lipschitz continuity in $u$ of $g(u)$, the assumption $g^0\in\bH^{\gamma}_{d+4,d-1/2,(2)}$, and that by Theorem \ref{thm:intcont},
\begin{equ}
u\in\frH^{\gamma}_{d+4,d-2+d+4,(2)}\subset\bH^{\gamma}_{d+4,d-2,(2)}.
\end{equ}
\end{proof}

\subsection{An It\^o-Wentzell formula}
In light of Lemma \ref{lem:simplify}, we consider
\begin{equs}
du_t(x)&=(a^{ij}_t(x)D_iD_ju_t(x)+f_t(x))\,dt+\sigma^{ik}_t(x)D_iu_t(x)\,dW^k_t&\quad&\text{on }\R_+\times G,
\\
u_t(x)&=0&\quad&\text{on }\R_+\times\partial G,
\\
u_0(x)&=\psi(x)&\quad&\text{on } G.
\end{equs}
Consider the flow $(X_t(x))_{(t,x)\in Q^+}$ given by the SDE
\begin{equ}\label{eq:flow X}
dX_t=-\sigma^{k}_t(X_t)\,dW^k_t,
\end{equ}
which exists under Assumption \ref{as:coeff regularity} (a) by the general theory of stochastic flows, see for example \cite[Thm~II.3.1]{Kunita_StFleur}. Here and below $\sigma^k$ stands for the column vector $(\sigma^{1k},\ldots,\sigma^{dk})$.
Since the coefficients are assumed to vanish outside $G^+$, the flow $X$, and in fact any flow appearing below that is built from the coefficients $a$ and $\sigma$, are trivial outside $G^+$.
Formally applying the It\^o-Wentzell formula, the field $v_t(x):=u_t(X_t(x))$ is expected to satisfy
\begin{equs}[eq:transformed]
\partial_t v_t(x)&=
\bar a^{ij}_t(X_t(x))(D_iD_j u_t)(X_t(x))-(\sigma_t^{ik}D_i\sigma_t^{jk})(X_t(x))(D_ju_t)(X_t(x))
\\
&\quad+f_t(X_t(x))
\\
&=\alpha^{ij}_t(x)D_iD_j v_t(x)+\beta^i_t(x)D_iv_t(x)+\varphi_t(x),
\end{equs}
on the (random) domain
$$\tilde Q:=\{(t,x):\,t\in(0,T], X_t(x)\in G\},$$
 where
here and in the following we use the notations
\begin{equs}
\alpha_t(x)&=\alpha_t(x)(\omega)=(\nabla X_t(x))^{-1}\bar a_t(X_t(x))((\nabla X_t(x))^*)^{-1}
\\
\beta_t(x)&=\beta_t(x)(\omega)=(\nabla X_t(x))^{-1}\Big(\Sigma_t(X_t(x))-\nabla^2(X_t(x))\alpha_t(x)\Big)
\\
\varphi_t(x)&=\varphi_t(x)(\omega)=f_t(X_t(x))
\\
\Sigma_t(x)&=(\nabla\sigma_t(x))\sigma_t^*(x)
\end{equs}
Unfortunately no version of the It\^o-Wentzell formula known to the author is actually applicable here, so we should confirm that the above formal computation is correct. It is worth noting that (again due to \cite[Thm~II.3.1]{Kunita_StFleur}) the coefficients $\alpha$ and $\beta$ are both almost surely uniformly (in $t$) bounded in $\cC^{2+\nu/2}$ and $\cC^{1+\nu/2}$, respectively.
\begin{lemma}
Let Assumptions \ref{as:coercivity}, \ref{as:coeff regularity} (a), and \ref{as:special} hold. Then with the above notations, 
for almost all $\omega\in\Omega$, the function $(v_t(x))_{(t,x)\in\tilde Q(\omega)}(\omega)$ is the probabilistic solution of \eqref{eq:transformed} on $\tilde Q(\omega)$, with initial condition $\psi$ and boundary condition $0$.
\end{lemma}
\begin{proof}
Recall a Feynman-Kac formula for SPDEs with Dirichlet boundary conditions from \cite{GG_FK}. Let $(B_t^{r})_{r=1,\ldots,d,\,t\geq 0}$ be the canonical $d$-dimensional Wiener process on the standard Wiener space $(\hat\Omega,(\hat\cF_t)_{t\geq0},\hat P)$.
Fix for now and for the rest of the paper $\rho$ to be a $\cC^{2+\nu}(G)$ square root of $2\bar a$. Introducing the flow $Y$ given by the SDE given on the completion of the probability space $(\Omega\times\hat\Omega,(\cF_t\otimes\hat\cF_t)_{t\geq0},P\otimes\hat P)$,
\begin{equ}
dY_t=(\sigma^{ik}_tD_i\sigma^k_t+\rho^{ir}_tD_i\rho^r_t)(Y_t)\,dt-\sigma^k_t(Y_t)\,dW^k_t-\rho^r_t(Y_t)\,dB^r_t
\end{equ}
and the exit time of the inverse characteristics
\begin{equ}
\gamma_{t,x}=\sup\{s\in[0,t]:\,(s,Y_{s,t}^{-1}(x))\notin(0,T]\times G\},
\end{equ}
one has by  \cite[Thm~2.1]{GG_FK}, for all $t\in[0,T]$, $dx\otimes dP$-almost everywhere
\begin{equ}
u_t(x)=\E^{\hat P}\Big(\psi(Y_{0,t}^{-1}(x))\bone_{\gamma_{t,x}=0}+\int_{\gamma_{t,x}}^tf_s(Y_{s,t}^{-1}(x)))\,ds\Big).
\end{equ}
For a fixed $s\in[0,T]$, consider $w=(w^{(1)},\ldots,w^{(d)})$, the solution of the (system of) fully degenerate SPDEs
\begin{equ}
dw^{(l)}_t(x)=a^{ij}_t(x)D_iD_jw^{(l)}_t(x)\,dt+\sigma^{ik}_t(x)D_iw^{(l)}_t(x)\,dW^k_t+\rho^{ir}_t(x)D_iw^{(l)}_t(x)\,dB^r_t
\end{equ}
on $[s,T]\times\R^d$, with initial condition $w^{(l)}_s(x)=x^l$, $l=1,\ldots,d$, which exists and is unique by \cite{GGK14}. 
Now we may apply the It\^o-Wentzell formula \cite[Thm~1.1]{K_I-W} and verify that the differential of $w^{(l)}_t(Y_{s,t}(x))$ is $0$, and hence $w^{(l)}_t(x):=(Y^{-1}_{s,t}(x))^l$.
Applying the It\^o-Wentzell formula again, one sees that $z_t^{(l)}(x):=w_t^{(l)}(X_t(x))$ satisfies, with the notation $\bar \rho_t(x)=(\nabla X_t(x))^{-1}\rho_t(X_t(x))$,
\begin{equs}[eq:z]
dz^{(l)}_t(x)&=[\bar a_t^{ij}(X_t(x))D_iD_jw^{(l)}_t(X_t(x))-(\sigma_t^{ik}D_i\sigma^{jk}_t)(X_t(x))D_jw^{(l)}_t(X_t(x))]\,dt
\\&\quad\quad+\rho^{ir}_t(X_t(x))D_iw^{(l)}_t(X_t(x))\,dB_t^r
\\&=[\alpha^{ij}_t(x)D_iD_jz^{(l)}_t(x)+\beta^i_t(x)D_iz^{(l)}_t(x)]\,dt+\bar\rho^{ir}_t(x)D_iz^{(l)}_t(x)\,dB^r_t
\end{equs}
on $[s,T]\times\R^d$ with initial condition $z^{(l)}_s(x)=X_s^l(x)$.
Note that (due to again \cite[Thm~II.3.1]{Kunita_StFleur}) the coefficients $\alpha$, $\beta$, $\bar\rho$ are almost surely bounded processes in $\cC^{2+\nu/2}$, $\cC^{1+\nu/2}$, and $\cC^{2+\nu/2}$, respectively. So (see e.g. \cite{Krylov_Intro}) one can find processes $\beta^{[m]}$ and $\bar\rho^{[m]}$ which are step functions in the sense that they are of the form
\begin{equ}
\sum_{i=1}^{k}\sum_{j=1}^{l_i}\bone_{[t_{i-1},t_{i})}\bone_{A_j^i}h_{i,j}
\end{equ}
with some $k=k(m)$, $l_i=l_i(m)$, some partition $0=t_0<t_1<\cdots<t_{k}=T$, some $\cF_{t_{i-1}}$-measurable events $A_j^i$, and some deterministic smooth functions $h_{i,j}$, and such that furthermore
\begin{equ}\label{eq:conv}
\|(\beta^{[m]}-\beta)\|_{\cC^{1+\nu/3}}
+\|(\bar\rho^{[m]}-\bar\rho)\|_{\cC^{2+\nu/3}}\rightarrow0
\end{equ}
as $m\rightarrow 0$ in measure with respect to $dt\otimes dP$. One can of course also assume that the left-hand side above never exceeds $1$.
Denote by $z^{[m]}$ the solution of \eqref{eq:z} with $\alpha=\bar\rho\bar\rho^*$, $\beta$, and $\bar\rho$ replaced by $\bar\rho^{[m]}(\bar\rho^{[m]})^*$, $\beta^{[m]}$, and $\bar\rho^{[m]}$, respectively. 
If we set $\tau_n:=\inf\{t\geq0:\,|(\nabla X_t(x))^{-1}|+|D_k X_t(x)|\leq n,\,\forall |k|\leq 3\}\wedge T$, then up to $\tau_n$, the coefficients are bounded in the appropriate spaces, and the existence an uniqueness of such solution on $[0,\tau_n]$ follows again from \cite{GGK14}, along with the fact that
\begin{equ}\label{eq:conv2}
\|z^{[m]}-z\|_{L_q(\llbracket0,\tau_n\rrbracket,H^1_p)}\rightarrow 0
\end{equ}
as $m\rightarrow \infty$, for any $p,q\in[2,\infty)$.

Now introduce the flow $Z^{[m]}$ given by the SDE on the probability space $(\Omega\times\hat\Omega,(\cF_t\otimes\hat\cF_t)_{t\geq0},P\otimes\hat P)$
\begin{equ}\label{eq:Zm}
dZ^{[m]}_t=(-\beta^{[m]}_t+\bar\rho_t^{[m],ir}D_i\bar\rho_t^{[m],r})(Z^{[m]}_t)\,dt-\bar\rho^{[m],r}_t(Z^{[m]}_t)\,dB^r_t.
\end{equ}
For almost all fixed $\omega$, $Z^{[m]}(\omega)$, as a function of $\hat\omega$, $s$, $t$, $x$, is the flow given by the SDE \eqref{eq:Zm} on the probability space $(\hat \Omega,(\hat\cF_t)_{t\geq0},\hat P)$, with `deterministic' coefficients $\beta^{[m]}(\omega)$, $\bar\rho^{[m]}(\omega)$. 
Moreover, the convergence \eqref{eq:conv} (at least along a subsequence) holds for almost all $\omega$ in measure with respect to $dt$.
So by the limit theorems of flows (see e.g \cite{Kunita_Book97}), the limit $Z:=\lim Z^{[m]}$ exists (for example, in $\cC^{\nu/4}(Q^+)$), and is on the one hand the flow corresponding to the equation
\begin{equ}\label{eq:Z'}
dZ_t=(-\beta_t+\bar\rho_t^{ir}D_i\bar\rho_t^{r})(Z_t)\,dt-\bar\rho^r_t(Z_t)\,dB^r_t.
\end{equ}
on $(\Omega\times\hat\Omega,(\cF_t\otimes\hat\cF_t)_{t\geq0},P\otimes\hat P)$, and on the other hand, also on $(\hat \Omega,(\hat\cF_t)_{t\geq0},\hat P)$, for almost all $\omega\in\Omega$. One more application of the It\^o-Wentzell formula then yields that the differential of $z^{[m]}_t(Z^{[m]}_{s,t}(x))$ is $0$, that is, $z^{[m]}_t(Z^{[m]}_{s,t}(x))=X_s(x)$. After passing to the limit using \eqref{eq:conv2}, and using the fact that both sides are continuous in all arguments, we therefore obtain that almost surely for all $0\leq s\leq t\leq T$, $x\in \R^d$,
\begin{equ}
Y_{s,t}^{-1}(X_t(Z_{s,t}(x)))=X_s(x),\quad\text{or,}\quad Y_{s,t}^{-1}(X_t(x))=X_s(Z_{s,t}^{-1}(x)).
\end{equ}

By \cite[Thm~II.6.1]{Kunita_StFleur}, for each fixed $\omega$ the inverse flow of $ Z(\omega)$ can be given explicitly: $Z_{s,t}^{-1}(\omega)=U_{t,s}(\omega)$, where the flow $U=U(\omega)$ goes backwards in time and is given by the SDE (parametrized by $\omega\in\Omega$)
\begin{equ}\label{eq:flow U}
dU_t=\beta_t(U_t)\,dt+\bar\rho_t^r(U_t)\,d\overleftarrow{B}^r_t.
\end{equ}
Furthermore, almost surely
\begin{equs}
\tau_{t,x}:=\gamma_{t,X_t(x)}&=\sup\{s\in[0,t]:\,(s,X_s(Z_{s,t}^{-1}(x)))\notin(0,T]\times G\}
\\
&=\sup\{s\in[0,t]:\,(s,Z_{s,t}^{-1}(x))\notin\tilde Q\}
\\
&=\sup\{s\in[0,t]:\,(s,U_{t,s}(x))\notin\tilde Q\}
\end{equs}
is indeed the exit time of $U$ from $\tilde Q$. Hence
\begin{equs}
v_t(x)&=u_t(X_t(x))
\\&=\E^{\hat P}\Big(\psi(Y_{t}^{-1}(X_t(x)))\bone_{\gamma_{t,X_t(x)}=0}+\int_{\gamma_{t,X_t(x)}}^tf_s(Y_{s,t}^{-1}(X_t(x))))\,ds\Big)
\\
&=\E^{\hat P}\Big(\psi(U_{t,0}(x))\bone_{\tau_{t,x}=0}+\int_{\tau_{t,x}}^t\varphi_s(U_{t,s})\,ds\Big),
\end{equs}
and the right-hand side is indeed the probabilistic solution of \eqref{eq:transformed} with initial condition $\psi$ and boundary condition $0$. While the above equality is a priori only justified for all $t\in[0,T]$, $dx\otimes dP$-almost everywhere, since both sides are continuous in $(t,x)\in\tilde Q$, the equality holds $P$-almost surely for all $(t,x)\in\tilde Q$.
\end{proof}

\subsection{Krylov's square root law for inverse flows}
Define, for $(x_t)_{t\in[0,T]}\in\cC([0,T],V)$, where $V$ is some normed vector space, for $c\in(0,\infty)$, and $t\in[0,T]$  the quantity
\begin{equ}
N_n(x_\cdot,c,t)=\#\{k=0,\ldots,n:\underset{[t-2^{-k},t]}{\text{osc}}x_\cdot>c2^{-k/2}\},
\end{equ}
where the convention $x_t=x_0$ for $t\in[-1,0)$ is used. We will need a generalization of the following square root law.
\begin{theorem}[\cite{K_Onemore}]\label{thm:Onemore}
Let $(w_t)_{t\geq0}$ be a 1-dimensional Wiener process. Then for all $c,T\in(0,\infty)$, almost surely
\begin{equ}
\limsup_{n\rightarrow\infty}\sup_{t\in[0,T]}\frac{1}{n+1}N_n(w_\cdot,c,t)=\pi(c),
\end{equ}
with a deterministic function $\pi(c)$ that converges to $0$ as $c\rightarrow\infty$.
\end{theorem}

First we prove the following auxiliary lemma. For deterministic $\sigma$, similar estimates often appear in rough path theory, but we could not find a version that implies this form. We therefore provide a proof (using in fact less regularity requirement on $\sigma$ than in for example \cite[Prop~8.3]{FH}).
\begin{lemma}\label{lem:RP}
Let $\lambda\in(0,1/2)$. Let $\sigma$ be a bounded predictable process with values in $\cC^{1}(\R^d)$ that vanishes outside $G^+$ and such that $\|\sigma\|_{\cC^{\lambda}([0,T],L_\infty(\R^d))}$ has finite moments of any order.
Then with the flow $X$ given by \eqref{eq:flow X}, any $\eps>0$ and $p\geq 0$, 
$$
\E\Big(\sup_{s,t\in[0,T];\,y\in G^+}\frac{|X_{t}(y)-X_{s}(y)+\sigma^k_{s}(X_{s}(y))(W^k_{t}-W^k_{s})|}{|t-s|^{(1+2\lambda)/2-2\eps}}\Big)^p\leq N
$$
for a constant $N$ depending on $p$, $\lambda$, $\eps$, $d$, $T$, $G$, and the bounds on $\sigma$.
\end{lemma}
\begin{proof}
We apply Lemma \ref{lem:Kolm} with $V=\cC(G^+,\R^d)$, 
\begin{equs}
D_{s,t}&=X_{t}(\cdot)-X_{s}(\cdot)+\sigma^k_{s}(X_{s}(\cdot))(W^k_{t}-W^k_{s}),
\\
E_{s,s',t,t'}&=(\sigma^k_s(X_s(\cdot))-\sigma^k_{s'}(X_{s'}(\cdot)))(W^k_t-W^k_{t'}).
\end{equs}
Condition \eqref{eq:Kolm cond 1} is clearly satisfied. As for the bounds \eqref{eq:Kolm cond 2}, first, using also the usual version of the Kolmogorov continuity theorem, we can write, for any $p\geq 1$, (up to constants depending on $p$, $d$, and $G$)
\begin{equs}
\E&|D_{s,t}|^p=\E\Big|\sup_{y\in G^+}\int_s^t\sigma^k_r(X_r(y))-\sigma^k_s(X_s(y))\,dW^k_r\Big|^p
\\
&\leq\E\Big|\int_s^t\sigma^k_r(X_r(y_0))-\sigma^k_s(X_s(y_0))\,dW^k_r\Big|^p
\\
&\quad+\sup_{y,y'\in G^+}\frac{\E\Big|\int_s^t\sigma^k_r(X_r(y))-\sigma^k_r(X_r(y'))-\sigma^k_s(X_s(y))+\sigma^k_s(X_s(y'))\,dW^k_r\Big|^p}{|y-y'|^{d+1}}
\\
&\leq\E\Big|\int_s^t|\sigma_r(X_r(y_0))-\sigma_s(X_s(y_0))|^2\,dr\Big|^{p/2}\label{eq:intermediate0}
\\
&\quad+\sup_{y,y'\in G^+}\frac{\E\Big|\int_s^t|\sigma_r(X_r(y))-\sigma_r(X_r(y'))-\sigma_s(X_s(y))+\sigma_s(X_s(y'))|^2\,dr\Big|^{p/2}}{|y-y'|^{d+1}},
\end{equs}
where $y_0\in G^+$ is arbitrary. Fix $\eps'\in(0,1/2-\lambda)$ and denote
\begin{equ}
K=\|\sigma\|_{L_\infty(\Omega\times[0,T],\cC^{1}(\R^d))},\,\,
\eta_1=\|\sigma\|_{\cC^{\lambda}([0,T],\cC^{1}(\R^d))},\,\,
\eta_2=\|X\|_{\cC^{1/2-\eps'}([0,T],\cC^{1-\eps'}(G^+))}.
\end{equ} 
The latter random variable has finite moments of any order due to \cite[Thm~II.2.1]{Kunita_StFleur}. One has
\begin{equs}
|\sigma_r(X_r(y_0))-\sigma_s(X_s(y_0))|&\leq 
\eta_1|r-s|^{\lambda}+|\sigma_s(X_r(y_0))-\sigma_s(X_s(y_0))|
\\
&\leq \eta_1|r-s|^{\lambda}+K|X_r(y_0)-X_s(y_0)|\label{eq:intermediate1}
\\
&\leq \eta_1|r-s|^{\lambda}+K\eta_2|r-s|^{1/2-\eps'}
\\&\leq(\eta_1+KT\eta_2)|r-s|^\lambda.
\end{equs}
As for the other term, first, using the same bound as above, with $y_0$ replaced by $y$ and $y'$,
\begin{equs}
|\sigma_r(X_r(y))-\sigma_s(X_s(y))|+|\sigma_s(X_s(y'))-\sigma_r(X_r(y'))|&\leq 2(\eta_1+KT\eta_2)|r-s|^\lambda
\end{equs}
On the other hand, one also has
\begin{equs}
|\sigma_r(X_r(y))-\sigma_r(X_r(y'))|+|\sigma_s(X_s(y))-\sigma_s(X_s(y'))|\leq 2K\eta_2|y-y'|^{1-\eps'}.
\end{equs}
Therefore we also have
\begin{equ}\label{eq:intermediate2}
|\sigma_r(X_r(y))-\sigma_r(X_r(y'))-\sigma_s(X_s(y))+\sigma_s(X_s(y'))|
\leq
\eta_3|r-s|^{\lambda(1-\eps)}|y-y'|^{(1-\eps')\eps}
\end{equ}
with some random variable $\eta_3$ with finite moments of any order.
Choosing $p$ large enough so that $p(1-\eps')\eps\geq d+1$, the second term on the right-hand side of \eqref{eq:intermediate0} is bounded by a constant times
\begin{equ}
|s-t|^{p(1/2+\lambda(1-\eps))},
\end{equ}
and hence, combining this with \eqref{eq:intermediate1}, we get
\begin{equ}
\E|D_{s,t}|^p\leq C|s-t|^{p(1/2+\lambda(1-\eps))}.
\end{equ}
uniformly in $s$ and $t$. Moving to the second bound in \eqref{eq:Kolm cond 2}, we have
\begin{equ}
\E|E_{s,s',t,t'}|^p\leq|t-t'|^{p/2}\E^{1/2}\sup_{y\in G^+}|\sigma_s(X_s(y))-\sigma_{s'}(X_{s'}(y))|^{2p}.
\end{equ}
The second component on the right-hand side can be estimated exactly as above: the only difference is that since one does not integrate in time, there is no factor $|s-s'|^{p/2}$ appearing. One hence has
\begin{equ}
\E|E_{s,s',t,t'}|^p\leq C|t-t'|^{p/2}|s-s'|^{p\lambda(1-\eps)},
\end{equ}
and so one can set $\gamma=(1+2\lambda)/2-\eps$ in Lemma \ref{lem:Kolm}: for large enough $p$ the conditions on the exponents are satisfied and we get the claim.
\end{proof}

We can now prove the desired square root law.
\begin{lemma}\label{lem:sqrt}
Let $\sigma$ satisfy Assumption \ref{as:coeff regularity} and assume $d_1<\infty$. Then, with the flow $X$ given by \eqref{eq:flow X}, for all $c\in(0,\infty)$, almost surely
\begin{equ}
\limsup_{n\rightarrow\infty}\sup_{t\in[0,T]}
\frac{1}{n+1}N_n((\nabla X_t^{-1})^{-1}X^{-1}_\cdot,c,t)\leq\hat\pi(c)
\end{equ}
with a deterministic function $\hat \pi(c)$, that depends only on $K$ and $d_1$, and that converges to $0$ as $c\rightarrow\infty$,
where for each $t$, $(\nabla X_t^{-1})^{-1}X^{-1}_\cdot$ is viewed as a process with values in $\cC(G^+,\R^d)$.
\end{lemma}
\begin{proof}
First note that
\begin{equs}
A_{s,t}:&=\sup_{y\in G^+}|X_s(y)-X_t(y)|\leq K |W_t-W_s|
\\
&\quad+ |t-s|^{(1+\nu)/2}\sup_{s',t'\in[0,T];\,y\in G^+}\frac{|X_{t'}(y)-X_{s'}(y)+\sigma_{s'}(X_{s'}(y))(W_{t'}-W_{s'})|}{|t'-s'|^{(1+\nu)/2}}.
\end{equs}
Denote the second term on the right-hand side by $B_{s,t}$, and let
\begin{equs}
S^i_t(c)&=\{n\in\mathbb{N}:\,\underset{[t-2^{-n},t]}{\text{osc}}W^i_\cdot>c2^{-n/2}\},\quad\quad i=1,\ldots,d_1
\\
S^0_t(c)&=\{n\in\mathbb{N}:\,\sup_{s\in[t-2^{-n},t]}B_{s,t}>c2^{-n/2}\}.
\end{equs}
Since due to Lemma \ref{lem:RP}, $B_{s,t}\leq|t-s|^{(1+\nu)/2} \eta$ for all $s,t\in[0,T]$ with a finite random variable $\eta$, the quantity $\sup_{t\in[0,T]}\# S^0_t(c)$ is also a.s. finite. By Theorem \ref{thm:Onemore},
\begin{equ}
\limsup_{n\rightarrow\infty}\sup_{t\in[0,T]}\frac{1}{n+1}\#(S^i_t(c)\cap[0,n])=\pi(c).
\end{equ}
Note that one has
\begin{equs}
X_t^{-1}(x)-X_s^{-1}(x)&=X_t^{-1}(X_s(X_s^{-1}(x)))-X_t^{-1}(X_t(X_s^{-1}(x)))
\\
&=\nabla X_t^{-1}(x)(X_s(X_s^{-1}(x))-X_t(X_s^{-1}(x)))+C_{s,t}
\end{equs}
where one has the estimates
\begin{equs}
|(X_s(X_s^{-1}(x))-X_t(X_s^{-1}(x)))|\leq A_{s,t},
\end{equs}
\begin{equs}
|C_{s,t}|&\leq \sup_{(t,x)\in Q^+}|\nabla^2X_t^{-1}(x)||X_s(X_s^{-1}(x))-X_t(X_s^{-1}(x))|^2
\\&\leq A_{s,t}^2\sup_{(t,x)\in Q^+}|\nabla^2X_t^{-1}(x)|.
\end{equs}
Hence for all $t\in[0,T]$,
\begin{equ}
\sup_{s\in[t-2^{-n},t]}\sup_{x\in G^+}|(\nabla X_t^{-1}(x))^{-1}(X_t^{-1}(x)-X_s^{-1}(x))|\leq \sup_{s\in[t-2^{-n},t]}(A_{s,t}+A_{s,t}^2\xi)
\end{equ}
with some finite random variable $\xi$. So whenever $n\notin \cup_{i=0}^{d_1}S^i_t(c)$ and $c2^{-n/2}\leq 1/\xi$, 
the right-hand side above is bounded by $2(Kd_1+1)c2^{-n/2}$ for all $s\in[t-2^{-n},t]$, and so setting $\hat\pi(c)=d_1\pi(c/(4Kd_1+4))$ finishes the proof of the lemma.
\end{proof}
\begin{remark}\label{rem:d1}
The square root law above is the only instance in the proof where the assumption $d_1<\infty$ is used.
As mentioned in Remark \ref{rem:elso}, a slight extension is available: for any sequence $(c_i)_{i\in\mathbb{N}}$, if $n\notin S_t:=\cup_{i=1}^{\infty}S^i_t(c_i)$, then one has
\begin{equ}
\sup_{s\in[t-2^{-n},t]}|A_{s,t}-B_{s,t}|\leq 2^{-n/2}\sup_{x\in G^+}\sum_{i=1}^\infty\sigma_t^i(X_t(x))c_i.
\end{equ}
If $\tilde c_i\rightarrow\infty$ sufficiently fast so that $\sum_{i=1}^\infty\pi(\tilde c_i)<\infty$, then the upper density of $S_t$ with the choice $c_i=C\tilde c_i$ can be made arbitrarily small by choosing $C$ to be sufficiently large. 
Hence if the decay of $\sigma^i$ is so fast that $\sum_{i=1}^{\infty}\sigma^i\tilde c_i<\infty$ uniformly over all choice of arguments, then we obtain the square root law as before.
Note however that this decay condition is far stronger than requiring $\sigma$ to be in $l_2$, or even $l_1$.
\end{remark}

Our final lemma, which will essentially allow us to utilize the above square root law, is the following estimate for hitting probabilities.
\begin{lemma}\label{lem:hitting}
Let $p\geq0$ be an integer, $(\bar B_t)_{t\geq 0}$ be a $d$-dimensional Wiener process on a filtered probability space $(\bar\Omega,(\bar\cF_t)_{t\geq0},\bar P)$, and let 
\begin{equ}
\xi_s=\int_0^s b_{s'}\,ds'+\int_0^s a_{s'}\,d\bar B_{s'}
\end{equ}
with $a_t$ and $b_t$ being bounded predictable processes with values in $\R^{d\times d}$ and $\R^d$, respectively. Fix
$c\geq 1$, $r\geq 7c$, denote $Q_r^p:= [0,2^{-p}]\times B_{2^{-p/2}r}$, fix $(t,x)\in Q_r^{p+1}$, and assume that on $\{(s,\omega):\,(t+s,x+\xi_s)\in Q_r^p\}$ the bounds $|b_s(\omega)|\leq C2^{p/2}$ and
\begin{equ}
\delta I\leq a_s(\omega) a_s^*(\omega)\leq \Delta I
\end{equ}
hold, the latter in the sense of positive semidefinite matrices, with some $C,\delta,\Delta>0$. Let furthermore $n\in\R^d$ be a unit vector, and $A\subset Q_r^p$ be a closed set such that $\{(s,y):\langle y,n\rangle\geq c2^{-p/2}\}\cap Q_r^p\subset A$. Finally, set
\begin{equ}
\tau_{t,x}=\inf\{s> 0:\,(t+s,x+\xi_s)\in A\cup\partial Q_r^p\}.
\end{equ}

Then
one has
\begin{equ}\label{eq:prob est0}
P((t+\tau_{t,x},x+\xi_{\tau_{t,x}})\notin A)\leq \gamma(c,r,d,\delta, \Delta,C)<1.
\end{equ}
for some function $\gamma$, depending only on the indicated parameters.
Moreover, for fixed $c,r,d,\Delta$, there exists a $c_0$ such that for sufficiently small $\delta$ one has $1-\gamma>e^{-c_0/\delta}$.
\end{lemma}
\begin{proof}
By rotational symmetry we may assume that $n=(-1,0,\ldots0)$ and by Brownian rescaling we may assume $p=0$. 
It is also clear that if $A$ is replaced by $\tilde A_c\cap \bar Q_r^0$, where
\begin{equ}
\tilde A_c=\{(s,y):y\in A_c\}:=\{(s,y):y_1\leq-c\},
\end{equ}
both in the definition of $\tau$ and on the left-hand side of \eqref{eq:prob est0},
 then the left-hand side of \eqref{eq:prob est0} increases, so it suffices to prove the statement for $A=\tilde A_c\cap\bar Q_r^0$. One can also trivially assume $x\notin \text{int} A_c$, since otherwise the left-hand side of \eqref{eq:prob est0} is just 0.

Let $\varphi,\psi$ be smooth functions on $\R$ such that
\begin{equs}
\text{for}\quad |a|&\leq 5/7 r,&\quad\varphi(a)&=\frac{1}{c+\sqrt{r^2-a^2}},&\quad \psi(a)&=c;
\\
\text{for}\quad |a|&\geq 6/7 r,&\quad\varphi(a)&=1,&\quad \psi(a)&=c+1;
\\
\text{for}\quad |a|&\in [5/7 r,6/7 r],&\quad\varphi(a)&\in[\frac{1}{c+\sqrt{r^2-a^2}},1],&\quad \psi(a)&\in[c,c+1].
\end{equs}
Denote $\tilde y:=(y_2,\ldots,y_d)$, and introduce the function
$f(y)=\varphi(|\tilde y|)(y_1+\psi(|\tilde y|)),$
the process 
\begin{equ}
\hat\xi_s:=f(x)+\int_0^s\hat b_{s'}\,ds'+\int_0^s\hat a_{s'}\,d\bar B_{s'}:=f(x+\xi_s),
\end{equ}
and the stopping time
\begin{equ}
\hat\tau:=\inf\{s>0:\,(t+s,\hat \xi_s)\in [0,1]\times\{0,\,1\}\cup \{1\}\times[0,1]\}.
\end{equ}
By construction, on $B_r\setminus \text{int} A_c$, $f$ is nonnegative, and on $\{y:\,|y|=r,\,y_1\geq -c\}$, one has $f(y)\geq1$. Therefore
\begin{equ}\label{eq:inc1}
\{(t+\tau_{t,x},x+\xi_{\tau_{t,x}})\notin \tilde A_c\cap\bar Q_r^0\}\subset
\{(t+\hat\tau,\hat \xi_{\hat\tau})\notin [0,1]\times\{0\}\}.
\end{equ}
Note also that 
\begin{equ}
|\nabla f(y)|\geq|D_1 f(y)|\geq\inf_{a\in \R}\varphi(a)=\frac{1}{c+r},
\end{equ}
and so $|\hat a|^2$ is bounded from below $\tfrac{\delta}{(r+c)^2}$. 
It is also clear from It\^o's formula that there exists a $\hat C=\hat C(r,c,d)$ such that $\sup_{s\in[0,1]}|\hat b_{s}|\leq \hat C(C+\Delta)$. Next, we claim that there exists an $m=m(r,c)<1$ such that for $y\in B_{2^{-1/2}}$, $f(y)\leq m$. Indeed, we can write
\begin{equs}
\max_{y_1^2+|\tilde y|^2\leq r^2/2} f(y)
&=\max_{y_1^2+|\tilde y|^2\leq r^2/2}\frac{y_1+c}{c+\sqrt{r^2-|\tilde y|^2}}
\\&=\max_{y_1^2+|\tilde y|^2= r^2/2}\frac{y_1+c}{c+\sqrt{r^2-|\tilde y|^2}}
\\&=\max_{y_1^2\leq r^2/2}\frac{y_1+c}{c+\sqrt{r^2/2+ y_1^2}}=:\max_{y_1^2\leq r^2/2}g(y_1).
\end{equs}
Trivially $\lim_{y_1\rightarrow\pm\infty}g(y_1)=\pm 1$, so it suffices to show that $g'>0$. Direct calculation shows
\begin{equ}\label{eq:0}
g'(y_1)=\frac{c\sqrt{r^2/2+y_1^2}-cy_1+r^2/2}{\sqrt{r^2/2+y_1^2}(\sqrt{r^2/2+y_1^2}+c)^2}.
\end{equ}
If the numerator were 0 for some $y_1$, that would imply
\begin{equ}
c^2(r^2/2+y_1^2)=c^2y_1^2-cr^2y_1+r^4/4,
\end{equ}
which gives $y_1=\frac{r^2/4-c^2/2}{c}$, but since substituting this back to \eqref{eq:0} gives a positive quantity, we get the claim.

Let us now set $t_0=(1-m)/(2\hat C(C+\Delta))$, so that one has $\sup_{s\in[0,t_0]}|\int_0^s\hat b_{s'}\,ds'|\leq (1-m)/2$. Define
\begin{equs}
\tilde \xi_s:&=f(x)+\int_0^s\hat a_{s'}\,d\bar B_{s'},
\\
\tilde\tau:&=\inf\{s>0:\,(t+s,\tilde \xi_s)\in [0,t+t_0]\times\{-\tfrac{1-m}{2},\,\tfrac{1+m}{2}\}\cup \{t+t_0\}\times[0,1]\}
\end{equs}
and notice that
\begin{equ}\label{eq:inc2}
\{(t+\hat\tau,\hat \xi_{\hat\tau})\notin [0,1]\times\{0\}\}\subset
\{(t+\tilde\tau,\tilde \xi_{\tilde\tau})\notin [0,t+t_0]\times\{-\tfrac{1-m}{2}\}\}.
\end{equ}
The latter event is now in the scope of \cite[Lem~3.7]{K_Onemore}: the process whose hitting time we are considering is a $1$-dimensional continuous martingale with quadratic variation uniformly bounded from below, and the starting point $(t,f(x))$ is strictly separated from the right boundary $[0,1]\times\{\tfrac{1+m}{2}\}$.
From \eqref{eq:inc1}, \eqref{eq:inc2}, and the application of \cite[Lem~3.7]{K_Onemore} one thus has
\begin{equ}
1-P((t+\tau_{t,x},x+\xi_{\tau_{t,x}})\notin A)\geq 
P(\sup_{t\in[0,\tilde t_0]}w_t\leq a\delta^{-1/2},\,
\inf_{t\in[0,\tilde t_0]}w_t\leq b\delta^{-1/2}),
\end{equ}
where $w$ is a 1-dimensional Wiener process, and the numbers $\tilde t_0,a,-b>0$ depend on $r,c,d,\Delta,C$. The right-hand side is clearly positive and the lower bound $e^{-c_0/\delta}$ bound follows from standard properties of the Wiener process.
\end{proof}

\subsection{Proof of Theorem \ref{thm:special}}

Throughout the proof we work with a fixed $\omega\in \Omega$. By linearity, we can assume that $\psi,f\geq0$, and hence also $u,v\geq0$. We will throughout the proof often use the shorthand $z=(t,x)$.

Define $v^\eps$ as the probabilistic solution of \eqref{eq:transformed} on
\begin{equ}
\tilde Q^\eps:=\{z:\,t\in[0,T],(\zeta_{\eps^3}\ast X_{\cdot}(x))_t\in (G+B_\eps)\},
\end{equ}
with initial condition $\psi$ and boundary condition $0$, where $\zeta\in\cC_0^\infty(\R_+)$ and $\zeta_\eps(s)=\eps^{-1}\zeta(\eps^{-1}s)$. 
Simply by the uniform in $x$ $1/2-$ H\"older-continuity in time of $X$,
there exists an $\eps_0=\eps_0(\omega)$ such that for all $0<\eps<\eps_0$, 
one has $\tilde Q\subset \tilde Q^\eps$ and therefore $v\leq v^\eps$. Moreover, $v^\eps$ agrees with the classical solution of \eqref{eq:transformed} with the same initial-boundary conditions and therefore it is continuously differentiable in time and twice continuously differentiable in space on the closure of $\{z:\,t\in [T_0/2,T],\,x\in\tilde Q^\eps_t\}$, where $\tilde Q_t^\eps=\{x:\,(t,x)\in\tilde Q^\eps\}$.

Using Lemma \ref{lem:sqrt} and the notation in Lemma \ref{lem:hitting}, fix a $c_0$ such that $1/2>\hat\pi(c_0)=:1-\hat\pi$ and a $r_0\geq7(2^{1/4}c_0+1)$, and set $\gamma:=\gamma(2^{1/4}c_0+1,r_0,d,\kappa2^{1/2},K2^{1/2},1)\vee(1/2)<1$.
 Since $M(z):=\nabla X_t(x)$ is uniformly continuous and separated away from zero, there exists a $\delta_0=\delta_0(\omega)>0$ such that whenever $|z-z'|\leq \delta_0$, $$
2^{-1/4}I\leq M(z)M^{-1}(z')\leq2^{1/4}I.
$$ 
Take any $\bar t\in[T_0,T]$.
Let $p_1,p_2\ldots,$ be the nonnegative integers such that
\begin{equ}\label{eq:sqrt2}
\sup_{x\in G^+}\sup_{s,t\in[\bar t-2^{-p_i},\bar t]}\big|M(\bar t,X_{\bar t}^{-1}(x))(X^{-1}_t(x)-X^{-1}_s(x))\big|\leq c_02^{-p_i/2},
\end{equ}
and introduce, for integers $j\geq-\log_2 (1\wedge (T_0/2))$,
\begin{equ}
S(j):=\{z:\,t\in[\bar t-2^{-j},\bar t],\, x\in\tilde Q_t,\,d(M(z)x,M(z)\partial\tilde Q_{\bar t})\leq r_0 2^{-j/2}\}
\end{equ}
and $\cM^\eps(j):=\sup_{S_j}|v^\eps|$, where for brevity we suppress the $\bar t$-dependence of these objects.
Of course $(\cM^\eps(j))$ is a decreasing sequence. Suppose now that there exists $\bar t$-independent indices $j_0=j_0(\omega)$ and $j_1=j_1(\eps,\omega)$ such that $j_1\rightarrow \infty$ as $\eps\rightarrow 0$ almost surely and that for all $j_1\geq p_i\geq j_0$
\begin{equ}\label{eq:geom0}
\cM^\eps(p_{i+1})\leq 2^{-p_i}\cK_2+ \gamma \cM^\eps(p_i).
\end{equ}
By iterating the above we get
\begin{equ}
\cM^\eps(p_{i+1})\leq \gamma^{i}\cK_2+\gamma \cM^\eps(p_i)\leq \gamma^{i-j_0}(i\cK_2+ \cM^\eps(0))\leq C\bar\gamma^{i}\cK_2(2+T).
\end{equ}
with $\bar \gamma=\gamma/2+1/2$ and some $C=C(\gamma,\omega)$. Denote by $j_2=j_2(\omega)$ the index such that
for all $j\geq j_2$, 
\begin{equ}\label{eq:dens}
\#\{i:\, p_i\leq j\}\geq j\hat\pi/2,
\end{equ}
which exists and does not depend on $\bar t$ by the definition of $\hat\pi$. We therefore obtain, for $j_1\geq j\geq j_2$
\begin{equ}
\cM^\eps(j)\leq \cM^\eps(p_{j\hat\pi/2})\leq C(2+T)\cK_2(\bar\gamma^{\hat\pi/2})^{j}.
\end{equ}
Denote $\hat C=C(2+T)$, $\hat\gamma=\bar\gamma^{\hat\pi/2}$. Note that for any $x\in G$, with $\bar\mu:=\sup_{z,z'\in Q^+}|M(z)||M^{-1}(z')|$, one has the trivial bound
\begin{equs}
d(M({\bar t},X_{\bar t}^{-1}(x))X_{\bar t}^{-1}(x),M({\bar t},X_{\bar t}^{-1}(x))\partial\tilde Q_{\bar t})&\leq d(x,\partial G)\bar\mu
\\&=r_0 2^{-[2\log_2(\tfrac{r_0}{d(x,\partial G)\bar\mu})]/2}.
\end{equs}
If $2\log_2(\tfrac{r_0}{d(x,\partial G)\bar\mu})> j_2$, choose $\eps\leq\eps_0$ such that $j_1(\eps)\geq2\log_2(\tfrac{r_0}{d(x,\partial G)\bar\mu})$, so that we can write
\begin{equs}
u_{\bar t}(x)=v_{\bar t}(X_{\bar t}^{-1}(x))\leq v^\eps_{\bar t}(X_{\bar t}^{-1}(x))&\leq \cM^\eps(\lfloor 2\log_2(\tfrac{r_0}{d(x,\partial G)\bar\mu}) \rfloor)
\\
&\leq \hat C\cK_2\hat\gamma^{\lfloor 2\log_2(\tfrac{r_0}{d(x,\partial G)\bar\mu}) \rfloor}
\\
&\leq \hat C\cK_2\hat\gamma^{2\log_2\tfrac{r_0}{\bar\mu}-2} d(x,\partial G)^{-2\log_2\hat\gamma}.
\end{equs}
Note that - as claimed in the theorem - the exponent $\alpha:=-2\log_2\hat\gamma>0$ of the decay depends only on $\kappa,K,d,d_1$.
Moreover, the exponential (in $1/\kappa$) lower bound on $\alpha$ follows from the corresponding statement of Lemma \ref{lem:hitting}.
If $2\log_2(\tfrac{r_0}{d(x,\partial G)\bar\mu})\leq j_2$ we can use the trivial bound
\begin{equ}
u_{\bar t}(x)\leq\sup_{\tilde Q}v\leq\cK_2 (T+1)
\leq \cK_2 (T+1)\left(\frac{2^{j_2/2}\bar\mu}{r_0}\right)^\alpha d(x,\partial G)^\alpha.
\end{equ}
Since $\bar t$ was arbitrary, this yields the claim, so it would suffice to prove \eqref{eq:geom0}. By virtually the same argument, it is also enough (and will be more convenient) to prove
\begin{equ}\label{eq:geom}
\cM^\eps(p_{i+2})\leq 2^{-p_i}\cK_2+ \gamma \cM^\eps(p_{i-1}).
\end{equ}

Recall that for any bounded $\cC^1$ domain there exists a function $\eps_G(\alpha):\,(0,1)\rightarrow (0,\infty)$ such that for any $\alpha\in(0,1)$ and $x\in\partial G$ one has $B_{\eps_G(\alpha)}(x)\cap\{y:\,\langle\tfrac{y-x}{|y-x|},n_x\rangle\geq \alpha\}\subset G^c$, where $n_x$ is the normal derivative of $\partial G$ at $x$. Let then $j_0$ be the smallest integer such that for all $j\geq j_0$  
\begin{enumerate}[(a)]
\item 
$2r_02^{-j/2}\tilde\mu\leq1/(32r_0)$, where $\tilde\mu=\sup_{z,z'\in Q^+}|\nabla^2X_t(x)||M^{-1}(z)|^2$,
\item $2\bar\mu r_0\bar 2^{-j/2}\leq\eps_G(1/16r_0)$,
\item $2^{-j}+\sup_{|s-t|\leq 2^{-j+1}}\sup_{y\in\R^d}| X_t^{-1}X_sy-y|+ \bar\mu r_02^{-j/2}\leq\delta_0,$
\item $\bar\mu\sup_{(t,x)\in \tilde Q^+}|\beta_t(x)|\leq
2^{-j/2}$.
\end{enumerate}
That is of course equivalent to saying that $j_0$ is the smallest integer that satisfies (a)-(d). Clearly, $j_0$ has no dependence on $\bar t$.

Fix now $i$ such that $p_i\geq j_0$ as well as $0<\eps<\eps_0$ and fix also $z_0=(t_0,x_0)\in S(p_{i+2})$. 
Let $z'=({\bar t},x')$, where $x'\in\partial\tilde Q_{\bar t}$ is a minimizer of the distance between $M(z_0)x_0$ and $M(z_0)\partial \tilde Q_{\bar t}$.
Recall the definition of the flow $U$ from \eqref{eq:flow U} and introduce,
with $M_0:=M(z')$
\begin{equ}
Q_0:=\{z:\,t\in[t_0-2^{-p_i},t_0],\,d(M_0x,M_0x')\leq r_02^{-p_{i}/2}\}
\end{equ}
\begin{equ}
A_0:=Q_0\setminus\tilde Q^\eps
\quad
\tau_0:=\sup\{s<t_0:\,(s,U_{t_0,s}(x_0))\in A_0\cup\partial Q_0\}.
\end{equ}
One has $z_0\in Q_0$, in fact, even
\begin{equ}\label{eq:00}
|M_0x_0-M_0x'|\leq 2^{1/4} |M(z_0)x_0-M(z_0)x'|\leq r_02^{-p_{i+2}/2+1/4}\leq r_02^{-(p_{i+1})/2}.
\end{equ}
Since $v^\eps$ is sufficiently smooth on the closure of $Q_0\cap \tilde Q^\eps$, by It\^o's formula one has
\begin{equs}
v^\eps_{t_0}(x_0)&=\E^{\hat P}\left(\int^{t_0}_{\tau_0} \varphi_s(U_{t_0,s}(x_0))\,ds+v^\eps_{\tau_0}(U_{t_0,\tau_0}(x_0))\right).
\end{equs}
If $z_{\tau_0}:=(\tau_0,U_{t_0,\tau_0}(x_0))\in A_0$, then $v^\eps(z_{\tau_0})=0$. If however $z_{\tau_0}\notin A_0$, then we claim that $z_{\tau_0}\in S(p_{i-1})$.
Indeed, first note that since one cannot exit $\tilde Q$ without exiting $\tilde Q_\eps$ first, if $z_{\tau_0}\notin A_0$, then one has $U_{t_0,\tau_0}(x_0)\in\tilde Q_{\tau_0}$.
Next, by property (c) of $j_0$, one has $|z_{\tau_0}-z'|\leq \delta_0$,
and hence
\begin{equs}
d(M(z_{\tau_0})U_{t_0,\tau_0}(x_0), M(z_{\tau_0})\partial\tilde Q_{\bar t})
&\leq 2^{1/4}
d(M_0U_{t_0,\tau_0}(x_0), M_0\partial\tilde Q_{\bar t})
\\
&\leq 2^{1/4}
|M_0U_{t_0,\tau_0}(x_0)- M_0x'|
\\
&\leq
2^{1/4}r_02^{-p_i/2}\leq r_02^{-p_{i-1}/2}.
\end{equs}
As for the time-coordinate, one simply has
\begin{equ}
\tau_0\geq t_0-2^{-p_i}\geq \bar t-2^{-p_i}-2^{-p_i}\geq \bar t-2^{-p_{i-1}},
\end{equ}
as required.
Hence,
\begin{equ}
v^\eps_{t_0}(x_0)\leq 2^{-p_i}\cK_2+\hat P((\tau_0,U_{t_0,\tau_0}(x_0))\notin  A_0)\cM^\eps(p_{i-1}).
\end{equ}
We now want to estimate the probability appearing on the right-hand side by $\gamma$, which is indeed enough to infer \eqref{eq:geom}. First let us transform the whole space by $M_0$:
\begin{equ}
Q_1:=(\text{id},M_0)Q_0,
\quad
 A_1:=(\text{id},M_0) A_0,
\end{equ}
and note that $\tau_0=\sup\{s<t_0:\,(s,M_0U_{t_0,s}(x_0))\in A_1\cup\partial Q_1\}$.

Let us now apply Lemma \ref{lem:hitting}
 with the following choice of parameters: 
\begin{itemize}
\item $p=p_i,\;\;r=r_0,\;\;c=2^{1/4}c_0+1$
\item $(\bar B_t)_{t\geq 0}=(B_{t_0-t}-B_{t_0})_{t\geq 0},\;\;(\bar\Omega,(\bar\cF_t)_{t\geq 0},\bar P)=(\hat\Omega,(\sigma((\bar B_s)_{s\in[0,t]})_{t\geq 0},\hat P)$ 
\item $A=\{(t,x):\,(-t,x)\in  A_1-(t_0,M_0x')\},\;\;$
$n=n_{X_{\bar t}(x')}$
\item $(t,x)=(0,M_0x_0-M_0x')$
\item $\xi_s=M_0U_{t_0,t_0-s}(x_0)-M_0x_0$
$$
=\int_0^sM_0\beta_{t_0-s'}(U_{t_0,t_0-s'}(x_0))\,ds'+\int_0^sM_0\bar\rho_{t_0-s'}(U_{t_0,t_0-s'}(x_0))\,d\bar B_{s'}  
$$
\item $\delta=\kappa2^{1/2},\;\;\Delta=K2^{1/2}$
\end{itemize} 
Let us verify the assumptions of Lemma \ref{lem:hitting}. 
The measurability conditions are satisfied by construction. 
The bound on the drift is satisfied due to property (d) of $j_0$. 
Concerning the bounds on the diffusion, first note that as seen above, property (c) of $j_0$ implies that whenever $(s,U_{t_0,s}(x_0))\in Q_0$, $|z'-(s,U_{t_0,s}(x_0))|\leq\delta_0$, and so
\begin{equs}
M_0\bar\rho_s(U_{t_0,s}(x_0))&=M_0(\nabla X_{s}(U_{t_0,s}(x_0)))^{-1}\rho_s(X_s(U_{t_0,s}(x_0)))
\\
&=M(z')M^{-1}((s,U_{t_0,s}(x_0)))\rho_s(X_s(U_{t_0,s}(x_0))),
\end{equs}
and the definition of $\delta_0$ along with the assumed bounds on $2\bar a=\rho\rho^*$ implies the claimed bounds. The condition on $(t,x)$ is straightforward and follows from \eqref{eq:00}.

As for the condition on $A$, first note that with denoting $\bar x:=X_{\bar t}(x')\in\partial G$, 
\begin{equs}
\cR:&=\{y:\,\scal{y-M_0x',n}\geq 2^{-p_i/2-1}\}\cap B_{2r_02^{-p_i/2}}(M_0x')
\\&\subset
\{y:\,\scal{\frac{y-M_0x'}{|y-M_0x'|},n}\geq 1/(4r_0)\}\cap B_{2r_02^{-p_i/2}}(M_0x')
\\&=
\{y:\,\scal{\frac{y-M_0X_{\bar t}^{-1}\bar x}{|y-M_0X_{\bar t}^{-1}\bar x|},n}\geq 1/(4r_0)\}\cap B_{2r_02^{-p_i/2}}(M_0X_{\bar t}^{-1}\bar x)
\end{equs}
Denoting further $\tilde x:=M_0X_{\bar t}^{-1}\bar x$, since one has $\nabla(X_{\bar t}M_0^{-1})(\tilde x)=I$, each $y$ in the latter set satisfies
\begin{equ}
X_{\bar t}M_0^{-1}y-\bar x=y-\tilde x+e,
\end{equ}
where 
\begin{equ}
|e|\leq |y-\tilde x|^2\tilde \mu\leq|y-\tilde x| 2r_02^{-p_i/2}\tilde\mu\leq|y-\tilde x|/(32 r_0)
\end{equ}
by property (a) of $j_0$. In particular, a very crude application of this bound implies
\begin{equ}
|y-\tilde x|/2\leq|y-\tilde x+e|\leq2|y-\tilde x|,
\end{equ}
and hence
\begin{equs}
\langle\frac{X_{\bar t}M_0^{-1}y-\bar x}{|X_{\bar t}M_0^{-1}y-\bar x|},n\rangle&\geq
\langle\frac{y-\tilde x}{|y-\tilde x+e|},n\rangle-\frac{|e|}{|y-\tilde x+e|}
\\
&\geq\frac{1}{2}\langle\frac{y-\tilde x}{|y-\tilde x|},n\rangle-\frac{1}{16r_0}\geq\frac{1}{16r_0}.
\end{equs}
Hence we can write
\begin{equs}
\cR&\subset
M_0X_{\bar t}^{-1}\Big(\{y:\,\scal{\frac{y-\bar x}{|y-\bar x|},n}\geq(1/16r_0)\}\cap B_{2\bar\mu r_02^{-p_i/2}}(\bar x)\Big)
\\&\subset
M_0X_{\bar t}^{-1}(G^c)
\end{equs}
where for the last inclusion we used property (b) of $j_0$. Let us now take an arbitrary
\begin{equ}
y^*\in\{y:\,\scal{y-M_0x',n}\geq (2^{1/4}c_0+1/2)2^{-p_i/2}\}\cap B_{(r_0+1)2^{-p_i/2}}(M_0x')
\end{equ}
and an $s\in[t_0-2^{-p_i},t_0]$. Denote $\bar y:=X_s M_0^{-1} y^*$ and $\tilde y:=M_0X_{\bar t}^{-1}\bar y$. Then one has
\begin{equs}[eq:01]
\tilde y- y^*&=M_0(X_{\bar t}^{-1}\bar y-X_s^{-1}\bar y)
\\&=
M_0M^{-1}(\bar t,X_{\bar t}^{-1}\bar y)M(\bar t,X_{\bar t}^{-1}\bar y)(X_{\bar t}^{-1}\bar y-X_s^{-1}\bar y).
\end{equs}
We have 
\begin{equs}
|X_{\bar t}^{-1}\bar y-x'|&\leq|X_{\bar t}^{-1} X_sM_0^{-1}y^*-M_0^{-1}y^*|+|M_0^{-1}y^*-x'|
\\
&\leq \sup_{y\in\R^d}|X_{\bar t}^{-1} X_sy-y|+\bar \mu r_02^{-p_i/2}\leq\delta_0
\end{equs}
by property (c) of $j_0$.
Hence, using the defining property \eqref{eq:sqrt2} of $p_i$, from \eqref{eq:01} we get
\begin{equ}
|\tilde y- y^*|\leq 2^{1/4}c_02^{-p_i/2}.
\end{equ}
It follows that $\tilde y\in \cR\subset M_0X_{\bar t}^{-1}(G^c)=M_0\tilde Q_{\bar t}^c$, and so 
$$
y^*=M_0X_s^{-1}X_{\bar t}M_0^{-1}\tilde y\in M_0\tilde Q_s^c.
$$
 Hence
\begin{equs}
\,[t_0&-2^{-p_i},t_0]&
\\&
\times\Big(\{y:\,\scal{y-M_0x',n}\geq (2^{1/4}c_0+1/2)2^{-p_i/2}\}\cap &B_{(r_0+1)2^{-p_i/2}}(M_0x')\Big)
\\&&\subset(\text{id},M_0)\tilde Q^c.
\end{equs}
Let now $j_1=j_1(\eps,\omega)$ be the largest integer such that the Hausdorff distance between $\tilde Q$ and $\tilde Q^\eps$ is smaller than $2^{-j_1-1}$. Then if $p_i\leq j_1$, we get
\begin{equs}
\,[t_0&-2^{-p_i},t_0]&&
\\&\times\Big(\{y:\,\scal{y-M_0x',n}\geq (2^{1/4}c_0+1)2^{-p_i/2}\}\cap
&&B_{r_02^{-p_i/2}}(M_0x')\Big)
\\&&&\subset(\text{id},M_0)(Q_0\setminus \tilde Q^\eps).
\end{equs}
By a simple translation and reflection we get the desired property of $A$.
Also notice that
$\tau_{t,x}=t_0-\tau_0$ and 
\begin{equ}
\{(\tau_0,U_{t_0,\tau_0}(x_0))\notin  A_0\}=
\{(\tau_0,M_0U_{t_0,\tau_0}(x_0))\notin  A_1\}=
\{(\tau_{t,x},x+\xi_{\tau_{t,x}})\notin A\}.
\end{equ}
Applying Lemma \ref{lem:hitting} therefore yields
\begin{equ}
\hat P((\tau_0,U_{t_0,\tau_0}(x_0))\notin  A_0)\leq \gamma
\end{equ}
as claimed and the proof is concluded.
\qed

\begin{appendices}
\section[Appendix]{Appendix}
The following lemma is a variation on Kolmogorov's H\"older-estimate, with the difference being that the two-parameter family we estimate here is not necessarily represented as increments of a (one-parameter) function.
\begin{lemma}\label{lem:Kolm}
Let $(V,|\cdot|)$ be a normed vector space and let $(D_{s,t})_{s,t\in[0,T]}$ and $(E_{s,s',t,t'})_{s,s',t,t'\in[0,T]}$ be two families of $V$-valued random variables, satisfying 
\begin{equs}[eq:Kolm cond 1]
|D_{s,t}|&\leq|D_{s,r}|+|D_{r,t}|+|E_{s,r,r,t}|
\\
|E_{s_1,s_2,s_3,s_4}|&\leq(|E_{s_1,t,s_3,s_4}|+|E_{t,s_2,s_3,s_4}|)\wedge(|E_{s_1,s_2,s_3,t}|+|E_{t,s_2,t,s_4}|)
\end{equs}
for all choice of arguments. Suppose furthermore that $D$ is almost surely continuous
in both arguments and that for some $p\geq1$, $C>0$, $\alpha, \alpha_1,\alpha_2>0$ the bounds
\begin{equs}[eq:Kolm cond 2]
\E|D_{s,t}|^p&\leq C|s-t|^{\alpha}
\\
\E|E_{s,s',t,t'}|^p&\leq C|s-s'|^{\alpha_1}|t-t'|^{\alpha_2}
\end{equs}
hold uniformly in $s,s',t,t'$. Then, if $0<p\gamma<(\alpha-1)\wedge(\alpha_1+\alpha_2-2)$, then
\begin{equ}\label{eq:Kolm}
\E\Big(\sup_{s\neq t\in[0,T]}|s-t|^{-\gamma}|D_{s,t}|\Big)^p\leq C N(T,\gamma,\alpha,\alpha_1,\alpha_2,p).
\end{equ}
\end{lemma}

\begin{proof}
We assume without loss of generality $T=1$. Introduce the notations $\cD_k=2^{-k}\mathbb{Z}\cap[0,1]$ and $\cD=\cup_{k=0}^\infty\cD_k$ for the dyadic numbers and note that due to the continuity of $D$, it suffices to take supremum over $s,t\in\cD$ in \eqref{eq:Kolm}. For fixed $s,t\in\cD$ 
let $n\in\mathbb{N}$ be such that $2^{-n-1}\leq|s-t|\leq 2^{-n}$. 
Let $(s_k)_{k\geq n}$ and $(t_k)_{k\geq n}$ be two sequences such that $s_k,t_k\in\cD_k$, $|s_n-t_n|\leq 2^{-n}$, $|s_{k+1}-s_k|\vee|t_{k+1}-t_k|\leq 2^{k+1}$, 
and that for some large enough $N$, $|s_k-s|\vee|t_k-t|=0$ for all $k\geq N$. One then has, due to \eqref{eq:Kolm cond 1},
\begin{equs}
|D_{s,t}|&\leq|D_{s,s_n}|+|D_{s_n,t_n}|+|D_{t_n,t}|+|E_{s,s_n,s_n,t}|+|E_{s_n,t_n,t_n,t}|
\\
&\leq \sum_{k=n}^{N}|D_{s_{k+1},s_k}|+\sum_{k=n}^{N}|E_{s,s_{k+1},s_{k+1},s_k}|+|D_{s_n,t_n}|
\\&\quad +\sum_{k=n}^{N}|D_{t_{k},t_{k+1}}|+\sum_{k=n}^{N}|E_{t_k,t_{k+1},t_{k+1},t}|+|E_{s,s_n,s_n,t}|+|E_{s_n,t_n,t_n,t}|
\\
&=:\sum_{i=1}^7 I_i.
\end{equs}
Clearly each of $I_1$, $I_3$, and $I_4$ is bounded (up to a universal constant) by 
\begin{equ}
2^{-\gamma n}\sup_{k\geq 0}\sup_{r\in \cD_k}|D_{r,r+2^{-k}}|2^{\gamma k}=:2^{-\gamma n}A.
\end{equ}
Choose $\gamma_1, \gamma_2>0$ such that $\gamma_1+\gamma_2=\gamma$ and $p\gamma_i<\alpha_i-1$ for $i=1,2$. 
Then each of $I_2$, $I_5$, $I_6$, and $I_7$ is bounded (up to a universal constant) by
\begin{equ}
2^{-\gamma n}\sup_{k,k'\geq 0}\sup_{\substack{r\in\cD_k \\ r'\in\cD_{k'}}}|E_{r,r+2^{-k},r',r'+2^{-k'}}|2^{\gamma_1 k}2^{\gamma_2 k'}=:2^{-\gamma n}B.
\end{equ}
This can be easily seen, for example in the case of $I_2$ (the other terms can be treated similarly), from
\begin{equs}
I_2\leq\sum_{k=n}^{N}\sum_{k'=k+1}^N|E_{s_{k'+1},s_{k'},s_{k+1}, s_k}|\leq B\sum_{k=n}^{\infty}\sum_{k'=k+1}^\infty 2^{-k'\gamma_1}2^{-k\gamma_2}\leq B2^{-n(\gamma_1+\gamma_2)}.
\end{equs}
Therefore 
\begin{equ}
\E\Big(\sup_{s\neq t\in[0,T]}|s-t|^{-\gamma}|D_{s,t}|\Big)^p\leq 7^p\E(A\vee B)^p,
\end{equ}
and it remains to bound $\E(A\vee B)^p$. Using the bounds \eqref{eq:Kolm cond 2} and the conditions on the exponents, one has, up to constants depending on  $p$ and the exponents,
\begin{equs}
\E&(A\vee B)^p
\\&\leq \E\sup_{k,k'\geq 0}\sup_{\substack{r\in\cD_k \\ r'\in\cD_{k'}}}
|D_{r,r+2^{-k}}|^p2^{\gamma kp}+
|E_{r,r+2^{-k},r',r'+2^{-k'}}|^p2^{\gamma_1 kp}2^{\gamma_2 k'p}
\\
&\leq \sum_{k\geq 0}\sum_{r\in\cD_k}
\E|D_{r,r+2^{-k}}|^p2^{\gamma kp}
+\sum_{k,k'\geq 0}\sum_{\substack{r\in\cD_k \\ r'\in\cD_{k'}}}
\E|E_{r,r+2^{-k},r',r'+2^{-k'}}|^p2^{\gamma_1 kp}2^{\gamma_2 k'p}
\\
&\leq C\sum_{k\geq0}2^k2^{-\alpha k}2^{\gamma k p}+C\sum_{k,k'\geq 0}2^{k+k'}2^{-\alpha_1 k}2^{-\alpha_2 k'}
2^{\gamma_1 kp}2^{\gamma_2 k'p}
\leq C.
\end{equs}
\end{proof}

\end{appendices}

\bibliography{dirichlet}{}

\begin{thebibliography}{Kim04b}
\expandafter\ifx\csname url\endcsname\relax
  \def\url#1{\texttt{#1}}\fi
\expandafter\ifx\csname urlprefix\endcsname\relax\def\urlprefix{URL }\fi
\expandafter\ifx\csname href\endcsname\relax
  \def\href#1#2{#2}\fi
\expandafter\ifx\csname burlalt\endcsname\relax
  \def\burlalt#1#2{\href{#2}{\texttt{#1}}}\fi

\bibitem[DG14]{DGy14}
\textsc{K.~A. Dareiotis} and \textsc{I.~Gy{\"o}ngy}.
\newblock {A Comparison Principle for Stochastic Integro-Differential
  Equations}.
\newblock \emph{Potential Analysis} \textbf{41}, no.~4, (2014), 1203--1222.
\newblock
  \burlalt{doi:10.1007/s11118-014-9416-7}{http://dx.doi.org/10.1007/s11118-014-9416-7}.

\bibitem[DG15]{DG15}
\textsc{K.~Dareiotis} and \textsc{M.~Gerencs\'er}.
\newblock On the boundedness of solutions of {SPDE}s.
\newblock \emph{Stochastic Partial Differential Equations: Analysis and
  Computations} \textbf{3}, no.~1, (2015), 84--102.
\newblock
  \burlalt{doi:10.1007/s40072-014-0043-5}{http://dx.doi.org/10.1007/s40072-014-0043-5}.

\bibitem[FH14]{FH}
\textsc{P.~K. Friz} and \textsc{M.~Hairer}.
\newblock \emph{A course on rough paths}.
\newblock Universitext. Springer, Cham, 2014.
\newblock With an introduction to regularity structures.
\newblock
  \burlalt{doi:10.1007/978-3-319-08332-2}{http://dx.doi.org/10.1007/978-3-319-08332-2}.

\bibitem[Fla90]{Flan_Compatib}
\textsc{F.~Flandoli}.
\newblock Dirichlet boundary value problem for stochastic parabolic equations:
  compatibility relations and regularity of solutions.
\newblock \emph{Stochastics and Stochastic Reports} \textbf{29}, no.~3, (1990),
  331--357.
\newblock
  \burlalt{doi:10.1080/17442509008833620}{http://dx.doi.org/10.1080/17442509008833620}.

\bibitem[GG18]{GG_FK}
\textsc{M.~Gerencs{\'e}r} and \textsc{I.~Gy{\"o}ngy}.
\newblock {A Feynman–Kac formula for stochastic Dirichlet problems}.
\newblock \emph{Stochastic Processes and their Applications} (2018).
\newblock
  \burlalt{doi:https://doi.org/10.1016/j.spa.2018.04.003}{http://dx.doi.org/https://doi.org/10.1016/j.spa.2018.04.003}.

\bibitem[GGK14]{GGK14}
\textsc{M.~Gerencs\'er}, \textsc{I.~Gy\"ongy}, and \textsc{N.~V. Krylov}.
\newblock On the solvability of degenerate stochastic partial differential
  equations in {S}obolev spaces.
\newblock \emph{Stochastic Partial Differential Equations: Analysis and
  Computations} \textbf{3}, no.~1, (2014), 52--83.
\newblock
  \burlalt{doi:10.1007/s40072-014-0042-6}{http://dx.doi.org/10.1007/s40072-014-0042-6}.

\bibitem[Kim04a]{Kim_Divergence}
\textsc{K.-H. Kim}.
\newblock On {$L_p$} -theory of stochastic partial differential equations of
  divergence form in {$C^1$} domains.
\newblock \emph{Probability Theory and Related Fields} \textbf{130}, no.~4,
  (2004), 473--492.
\newblock
  \burlalt{doi:10.1007/s00440-004-0368-5}{http://dx.doi.org/10.1007/s00440-004-0368-5}.

\bibitem[Kim04b]{Kim_Lp}
\textsc{K.-H. Kim}.
\newblock On stochastic partial differential equations with variable
  coefficients in {$C^1$} domains.
\newblock \emph{Stochastic Processes and their Applications} \textbf{112},
  no.~2, (2004), 261 -- 283.
\newblock
  \burlalt{doi:10.1016/j.spa.2004.02.006}{http://dx.doi.org/10.1016/j.spa.2004.02.006}.

\bibitem[KL98]{KL_line}
\textsc{N.~V. Krylov} and \textsc{S.~V. Lototsky}.
\newblock A {S}obolev space theory of {SPDEs} with constant coefficients on a
  half line.
\newblock \emph{SIAM Journal on Mathematical Analysis} \textbf{30}, no.~2,
  (1998), 298--325.
\newblock
  \burlalt{doi:10.1137/S0036141097326908}{http://dx.doi.org/10.1137/S0036141097326908}.

\bibitem[KL99]{KL_hspace}
\textsc{N.~V. Krylov} and \textsc{S.~V. Lototsky}.
\newblock A {S}obolev space theory of {SPDEs} with constant coefficients in a
  half space.
\newblock \emph{SIAM Journal on Mathematical Analysis} \textbf{31}, no.~1,
  (1999), 19--33.
\newblock
  \burlalt{doi:10.1137/S0036141098338843}{http://dx.doi.org/10.1137/S0036141098338843}.

\bibitem[KR81]{KR_SEE81}
\textsc{N.~V. Krylov} and \textsc{B.~L. Rozovskii}.
\newblock Stochastic evolution equations.
\newblock \emph{Journal of Soviet Mathematics} \textbf{16}, no.~4, (1981),
  1233--1277.
\newblock
  \burlalt{doi:10.1007/BF01084893}{http://dx.doi.org/10.1007/BF01084893}.

\bibitem[Kry94]{K_W2m}
\textsc{N.~V. Krylov}.
\newblock A {$W^2_n$} -theory of the {D}irichlet problem for {SPDE}s in general
  smooth domains.
\newblock \emph{Probability Theory and Related Fields} \textbf{98}, no.~3,
  (1994), 389--421.
\newblock
  \burlalt{doi:10.1007/BF01192260}{http://dx.doi.org/10.1007/BF01192260}.

\bibitem[Kry01]{K_Traces}
\textsc{N.~Krylov}.
\newblock {Some Properties of Traces for Stochastic and Deterministic Parabolic
  Weighted Sobolev Spaces}.
\newblock \emph{Journal of Functional Analysis} \textbf{183}, no.~1, (2001), 1
  -- 41.
\newblock
  \burlalt{doi:http://dx.doi.org/10.1006/jfan.2000.3728}{http://dx.doi.org/http://dx.doi.org/10.1006/jfan.2000.3728}.

\bibitem[Kry02]{Krylov_Intro}
\textsc{N.~Krylov}.
\newblock \emph{Introduction to the Theory of Random Processes}.
\newblock Graduate studies in mathematics. American Mathematical Society, 2002.

\bibitem[Kry03a]{K_Brown}
\textsc{N.~V. Krylov}.
\newblock Brownian trajectory is a regular lateral boundary for the heat
  equation.
\newblock \emph{SIAM Journal on Mathematical Analysis} \textbf{34}, no.~5,
  (2003), 1167--1182.
\newblock
  \burlalt{doi:10.1137/S0036141002402980}{http://dx.doi.org/10.1137/S0036141002402980}.

\bibitem[Kry03b]{K_Onemore}
\textsc{N.~V. Krylov}.
\newblock One more square root law for brownian motion and its application to
  spdes.
\newblock \emph{Probability Theory and Related Fields} \textbf{127}, no.~4,
  (2003), 496--512.
\newblock
  \burlalt{doi:10.1007/s00440-003-0301-3}{http://dx.doi.org/10.1007/s00440-003-0301-3}.

\bibitem[Kry08]{K_Survey}
\textsc{N.~V. Krylov}.
\newblock A brief overview of the {$L_p$}-theory of {SPDEs}.
\newblock \emph{Theory of Stochastic Processes} \textbf{14(30)}, no.~2, (2008),
  71--78.

\bibitem[Kry10]{K_I-W}
\textsc{N.~V. Krylov}.
\newblock On the {I}t{\^o}--{W}entzell formula for distribution-valued
  processes and related topics.
\newblock \emph{Probability Theory and Related Fields} \textbf{150}, no.~1,
  (2010), 295--319.
\newblock
  \burlalt{doi:10.1007/s00440-010-0275-x}{http://dx.doi.org/10.1007/s00440-010-0275-x}.

\bibitem[Kun84]{Kunita_StFleur}
\textsc{H.~Kunita}.
\newblock \emph{Stochastic differential equations and stochastic flows of
  diffeomorphisms},  143--303.
\newblock Springer Berlin Heidelberg, Berlin, Heidelberg, 1984.
\newblock
  \burlalt{doi:10.1007/BFb0099433}{http://dx.doi.org/10.1007/BFb0099433}.

\bibitem[Kun97]{Kunita_Book97}
\textsc{H.~Kunita}.
\newblock \emph{Stochastic Flows and Stochastic Differential Equations}.
\newblock Cambridge Studies in Advanced Mathematics. Cambridge University
  Press, 1997.

\bibitem[Lot00a]{Lot_Dirichlet}
\textsc{S.~V. Lototsky}.
\newblock Dirichlet problem for stochastic parabolic equations in smooth
  domains.
\newblock \emph{Stochastics and Stochastic Reports} \textbf{68}, (2000),
  145--175.

\bibitem[Lot00b]{Lot_WeightedSpaces}
\textsc{S.~V. Lototsky}.
\newblock Sobolev spaces with weights in domains and boundary value problems
  for degenerate elliptic equations.
\newblock In \emph{Methods and Applications of Analysis},  195 -- 204. 2000.

\bibitem[Lot01]{Lot_Degen}
\textsc{S.~Lototsky}.
\newblock Linear stochastic parabolic equations, degenerating on the boundary
  of a domain.
\newblock \emph{Electron. J. Probab.} \textbf{6}, (2001), 14 pp.
\newblock \burlalt{doi:10.1214/EJP.v6-97}{http://dx.doi.org/10.1214/EJP.v6-97}.

\bibitem[Par75]{Pardoux1}
\textsc{{\'E}.~Pardoux}.
\newblock \emph{{Equations aux derivées partielles stochastiques non lineaires
  monotones. Etude de solutions fortes de type Ito}}.
\newblock Ph.D. thesis, Univ. Paris Sud., 1975.

\bibitem[RY04]{RYor}
\textsc{D.~Revuz} and \textsc{M.~Yor}.
\newblock \emph{Continuous Martingales and Brownian Motion}.
\newblock Grundlehren der mathematischen Wissenschaften. Springer Berlin
  Heidelberg, 3rd ed., 2004.

\bibitem[Tri95]{Triebel}
\textsc{H.~Triebel}.
\newblock \emph{Interpolation theory, function spaces, differential operators}.
\newblock Johann Ambrosius Barth Verlag, 1995.

\end{thebibliography}
\bibliographystyle{Martin}

\end{document}